\documentclass[11pt,a4paper,twoside,reqno]{amsart}
\pdfoutput=1
\title[Closed $\text{SL}(3,\mathbb{C})$-structures on nilmanifolds]{Closed  $\text{SL}(3,\mathbb{C})$-structures on nilmanifolds}
\author{Anna Fino}
\address[A. Fino]{Dipartimento di Matematica ``G. Peano'' \\ Universit\`a degli Studi di Torino\\ Via Carlo Alberto 10\\10123 Torino\\ Italy}
\email{annamaria.fino@unito.it}
\author{Francesca Salvatore}
\address[F. Salvatore]{Dipartimento di Matematica ``G. Peano'' \\ Universit\`a degli Studi di Torino\\ Via Carlo Alberto 10\\10123 Torino\\ Italy}
\email{francesca.salvatore@unito.it}


\usepackage{hyperref,url}
\usepackage{amssymb,latexsym}
\usepackage{amsmath,amsthm}
\usepackage[latin1]{inputenc}
\usepackage{enumerate}
\usepackage{float}
\usepackage{pifont}
\usepackage{mathtools}
\usepackage{stmaryrd}
\usepackage{hyperref}
\usepackage{xcolor}
\usepackage[italian]{varioref}
\makeatletter
\labelformat{equation}{\tagform@{#1}}
\makeatother
\usepackage{multirow}
\usepackage{graphicx}
\usepackage{enumitem}
\usepackage{hyphenat}

\theoremstyle{plain}
\newtheorem{theorem}{Theorem}[section]
\newtheorem*{theoremA}{Theorem A}
\newtheorem*{theoremB}{Theorem B}
\newtheorem*{theoremC}{Theorem C}
\newtheorem{lemma}[theorem]{Lemma}

\newtheorem{proposition}[theorem]{Proposition}

\theoremstyle{definition}
\newtheorem{definition}[theorem]{Definition}
\newtheorem{example}[theorem]{Example}
\theoremstyle{remark}
\newtheorem{remark}[theorem]{Remark}

\numberwithin{equation}{section} 

\newcommand{\cmark}{\ding{51}}%

\usepackage{scalerel,stackengine}
\stackMath
\newcommand\widehatt[1]{%
\savestack{\tmpbox}{\stretchto{%
  \scaleto{%
    \scalerel*[\widthof{\ensuremath{#1}}]{\kern-.6pt\bigwedge\kern-.6pt}%
    {\rule[-\textheight/2]{1ex}{\textheight}}%WIDTH-LIMITED BIG WEDGE
  }{\textheight}% 
}{0.5ex}}%
\stackon[1pt]{#1}{\tmpbox}%
}

\begin{document}

\begin{abstract}
In this paper we consider closed $\text{SL}(3,\mathbb{C})$-structures which are either mean convex or tamed by a symplectic form.
These  notions were introduced by Donaldson in relation to $\text{G}_2$-manifolds with boundary.
In particular, we  classify nilmanifolds which carry an invariant mean convex closed $\text{SL}(3,\mathbb{C})$-structure and those which admit an invariant mean convex half-flat $\text{SU}(3)$-structure.
We also prove that, if a solvmanifold admits  an invariant  tamed closed $\text{SL}(3,\mathbb{C})$-structure, then it also  has  an invariant  symplectic half-flat $\text{SU}(3)$-structure. 
\end{abstract}

\maketitle

\section{introduction}

An $\text{SL}(3,\mathbb{C})$-structure  on an oriented   manifold  of real dimension $6$  is  defined by a   definite real  $3$-form $\rho$,  i.e.   by  a  stable $3$-form $\rho$  inducing an almost complex structure $J_{\rho}$ (see \cite{Hitchin,Reichel}). 
We shall say that  the  $\text{SL}(3,\mathbb{C})$-structure $\rho$  is \emph{closed} if $d\rho=0$. 
As remarked in  \cite{Donaldson},   closed $\text{SL}(3,\mathbb{C})$-structures  obey an $h$-principle, since any hypersurface in $\mathbb R^7$  acquires a closed $\text{SL}(3,\mathbb{C})$--structure.

A special case    of closed $\text{SL}(3,\mathbb{C})$-structure is given by a  \emph{closed  $\text{\normalfont SU}(3)$-structure}, i.e.   by the data of an almost Hermitian structure $(J, g, \omega)$ and a $(3,0)$-form $\Psi$  of non-zero constant length   satisfying 
 $$ \quad \frac{i}{2} \Psi \wedge \overline \Psi =  \frac 23  \omega^3,  \quad  d  (\text{Re}(\Psi) )=0.$$ Indeed the $3$-form  $\rho =  \text{Re}(\Psi)$ defines a closed $\text{SL}(3,\mathbb{C})$-structure such that  $J_{\rho} = J$.

As shown in \cite{Donaldson}, a closed $\text{SL}(3,\mathbb{C})$-structure always determines a real $3$-form $ \hat  \rho \coloneqq J_{\rho} \rho$ such that $d \hat \rho $ is of type $(2,2)$ with respect to $J_{\rho}$. Moreover  $\hat{\rho}$ is the imaginary part of a complex $(3,0)$-form $\Psi$. We shall   say that   a closed $\text{SL}(3,\mathbb{C})$-structure  is  \emph{mean convex} if  the $(2,2)$-form $d \hat \rho$ is   semi-positive. Note that $J_{\rho}$  is integrable if and only if  $d (J_{\rho} \rho)  = 0.$  A special class of  mean convex closed  $\text{SL}(3,\mathbb{C})$-structures is given by \emph{nearly-K\"ahler} structures.  Indeed, a  nearly-K\"ahler structure  can be defined as an   $\text{SU}(3)$-structure $( \omega, \Psi)$  satisfying 
 the following conditions:
\[d\omega=-\frac{3}{2}\nu_0 \,  \text{Re}(\Psi), \quad d (\text{Im}(\Psi))=\nu_0 \,  \omega^2,
\] 
where $\nu_0\in \mathbb{R}-\{0\}$ and therefore, up  to a change of sign of $ \text{Re}(\Psi)$, we can suppose  $\nu_0>0$.  The  nearly-K\"ahler condition forces the induced  Riemannian metric $g$  to be  Einstein  and, up to now, very few examples of manifolds admitting complete nearly-K\"ahler structures are known \cite{Butruille, FoscoloHaskins, Gray1, Gray2, Nagy1, Nagy2}. 
More in general,  an $\text{SU}(3)$-structure $(\omega, \Psi)$ such that  $d (\text{Re}(\Psi)) =0$ and  $d (\omega \wedge  \omega)=0$ is called {\em{half-flat}}, see    for instance \cite{BelCorFreGoe, Bryant, ChiossiSalamon, Conti, CorLeiSchSH, FidMinTom, GibLuPopSte, Hitchin2, Larfors}  for general results on this types of structures.  In particular, every oriented hypersurface of a Riemannian $7$-manifold with holonomy in $\text{G}_2$  is naturally endowed with a half-flat $\text{SU}(3)$-structure   and, conversely, using the Hitchin flow equations, a  $6$-manifold with a real analytic half-flat $\text{SU}(3)$-structure  can be realized as a hypersurface of a 7-manifold with holonomy in $\text{G}_2$ \cite{Bryant, Hitchin2}.

Nilmanifolds,   i.e. compact quotients  $\Gamma \backslash G$ of   connected, simply connected, nilpotent Lie groups  $G$  by a  lattice  $\Gamma$,  provide a large class of   compact  $6$-manifolds admitting invariant  closed $\text{SL}(3,\mathbb{C})$-structures  \cite{ChiossiSalamon,ChiossiSwann,  Conti, ContiTomassini, FinoRaffero}, where by  invariant  we mean  induced by  a left-invariant one  on the nilpotent Lie group $G$.  Note that nilmanifolds cannot admit invariant nearly K\"ahler structures, since by \cite{Milnor} the Ricci tensor of a left-invariant metric  on  a non-abelian nilpotent Lie group always has 
a strictly negative direction and a strictly positive direction.

Since a  nilmanifold   is parallelizable, its  Stiefel-Whitney numbers and  Pontryagin numbers are   all zero, hence by  well-known theorems of Thom and Wall,   it   bounds orientably, i.e.  it  is diffeomorphic to the boundary of a compact  connected manifold $N$. So  it would be   a natural question to see if,  given a $6$-dimensional nilmanifold endowed with an invariant mean convex closed $\text{SL}(3,\mathbb{C})$-structure $\rho$,   there exists  on $N$  a closed  $\text{G}_2$-structure with boundary value   an  \lq \lq enhancement" of $\rho$   (see  \cite[Section 3.1]{Donaldson} for more details).

According to \cite{Goze, Magnin} there are 34 isomorphism classes  of  $6$-dimensional real nilpotent Lie algebras  $\mathfrak{g}_i$, $i =1, \ldots, 34,$  listed in Table  \ref{table1}.
In this paper we  classify   $6$-dimensional  nilpotent Lie algebras  admitting mean convex  closed $\text{SL}(3,\mathbb{C})$-structures:
\begin{theoremA} \label{theoremA}
Let $M= \Gamma \backslash G$ be a $6$-dimensional nilmanifold. Then $M$ admits invariant   mean convex closed  $\normalfont \text{SL}(3,\mathbb{C})$-structures if and only if  the Lie algebra  $\mathfrak{g}$ of $G$  is not isomorphic to any of the  six   Lie algebras $\mathfrak{g}_i$, $i=1,2,4,9,12,34$, as listed in Table \ref{table1}.
\end{theoremA} 
Using the classification of half-flat nilpotent Lie algebras (see \cite{Conti}), we can then prove the following:
\begin{theoremB}\label{half-flat mean convex}
A  nilmanifold $M= \Gamma \backslash G$ has  an invariant mean convex  half-flat  structure if and only if  the Lie algebra $\mathfrak{g}$ of $G$  is isomorphic to any of the Lie algebras $\mathfrak{g}_i$, $i=6,7,8,10,13,15,16,22,24,25,28,29,30,31,32,33$, as listed in Table \ref{table1}.
\end{theoremB}
Moreover, in Section \ref{Section6} we show that the mean convex condition is preserved by the Hitchin flow equations in some special cases. More generally, since in our examples the property is preserved for small times, it would be interesting to determine if this is always the case.

Given a closed  $\text{SL}(3,\mathbb{C})$-structure  $\rho$ on a $6$-manifold,  another  natural condition to  study  is the existence of a symplectic form $\Omega$ taming $J_{\rho}$, i.e. such that  $\Omega(X, J_{\rho} X)>0$ for each  non-zero vector field $X$. This is equivalent to  the  positivity  in the standard sense  of  the $(1,1)$-component   $\Omega^{1,1}$ of $\Omega$.
We  shall say that a closed $\text{SL}(3,\mathbb{C})$-structure  $\rho$ is \emph{tamed} if there exists a symplectic form $\Omega$ such that $\Omega^{1,1} >0$.

As shown in \cite{Donaldson} a mean convex $\text{SL}(3,\mathbb{C})$-structure on a compact $6$-manifold cannot be tamed by any  symplectic form.
If we remove the assumption of mean convexity,  examples of  tamed closed  $\text{SL}(3,\mathbb{C})$-structures are given by symplectic half-flat struc\-tures  $(\omega, \Psi)$, i.e., by  half-flat $\text{SU}(3)$-structures $(\omega,\Psi)$ with $d\omega=0$. In  this  case  $\rho=\text{Re}(\Psi)$  is tamed by the symplectic form $\omega$, since $\omega$ is of type $(1,1)$ with respect to $J_{\rho}$.
In \cite{ContiTomassini},  nilmanifolds  admitting invariant symplectic half-flat structures  were  classified. Later, this classification was generalized to solvmanifolds, i.e.  to compact quotients  $\Gamma \backslash G$ of   connected, simply connected, solvable Lie groups  $G$  by   lattices  $\Gamma$ (for more details, see \cite{SHFsolvmanifolds}).

We prove the following result:
\begin{theoremC}\label{tamed}
Let  $\Gamma \backslash G$ be a $6$-dimensional  solvmanifold, not a torus.  
Then $\Gamma \backslash G$   admits an invariant tamed  closed $\normalfont \text{SL}(3,\mathbb{C})$-structure  if and only if  the Lie algebra $\mathfrak{g}$  of $G$ has  symplectic half-flat structures.

If $\mathfrak{g}$  is nilpotent, then  it is isomorphic to $\mathfrak{g}_{24}$ or $\mathfrak{g}_{31}$ as listed in Table \ref{table1}. 

If $\mathfrak{g}$ is solvable,  then it is isomorphic to one among  $\mathfrak{g}_{6,38}^0$, $\mathfrak{g}_{6,54}^{0,-1}$, $\mathfrak{g}_{6,118}^{0,-1,-1}$, $\mathfrak{e}(1,1) \oplus \mathfrak{e}(1,1)$, $ A_{5,7}^{-1,\beta,-\beta}\oplus \mathbb{R}$,  $A_{5,17}^{0,0,-1} \oplus \mathbb{R}$, $ A_{5,17}^{\alpha,-\alpha,1} \oplus \mathbb{R}$, as listed in Table \ref{table3}.

Moreover, all the  nine  Lie algebras admit closed $\text{\normalfont SL}(3,\mathbb{C})$-structures tamed by a symplectic form $\Omega$ such that $d\Omega^{1,1}\neq 0$. 

\end{theoremC}
 
Explicit examples of closed $\text{SL}(3,\mathbb{C})$-structures tamed by a symplectic form $\Omega$ such that $d\Omega^{1,1}\neq 0$ are provided.
These examples provide new examples of closed $\text{G}_2$-structures on the product $M\times S^1$, where $M=\Gamma \backslash G$  is a $6$-dimensional solvmanifold endowed with an invariant tamed closed $\text{SL}(3,\mathbb{C})$-structure.
It would be interesting to see if there exist compact manifolds   which  have   tamed closed  $\text{SL}(3,\mathbb{C})$-structures   but do not admit any symplectic half-flat struc\-tures.

The paper is organized as follows. 
In Section \ref{section2} we review the general theory of semi-positive $(p,p)$-forms focusing on the case $p=2$. In Section \ref{section3} we study the  intrinsic torsion of closed $\text{SU}(3)$-structures in relation to the mean convex condition.

In Section \ref{section4} we prove Theorem A. Starting from this result, in Section \ref{section5} we prove Theorem B. In Section \ref{Section6} we study the behaviour of mean convex  half-flat $\text{SU}(3)$-structures under the Hitchin flow equations. Finally, in Section \ref{Section7}, we prove Theorem C.

{\it Acknowledgements.} 
The authors are supported by the  Project PRIN 2017 ``Real and complex manifolds:  Topology, Geometry and Holomorphic Dynamics''  and by G.N.S.A.G.A. of I.N.d.A.M. 
The authors would like to thank Simon Chiossi and Alberto Raffero for useful discussions and comments. The authors are also grateful to an anonymous referee for useful comments.

\section{Preliminaries on semi-positive differential forms}\label{section2}
In this section we review the definition and main results regarding semi-positive $(p,p)$-forms on complex vector spaces.  
For more details we refer for instance to \cite{Demailly,HarveyKnapp}.

Let $V$ be a complex vector space of complex dimension $n$, with coordinates $(z_1,\ldots,z_n)$. 
Note that $V$ can be  considered  also as a real vector space of dimension $2n$ endowed with the  complex structure $J$ given by the multiplication by $i$.

Consider the exterior algebra
\[
\Lambda V^*\otimes \mathbb{C} = \bigoplus  \Lambda^{p,q}V^*,
\]
where $\Lambda^{p,q}V^*$ is a shorthand for $\Lambda^pV^*\otimes \Lambda^q \overline{V}^*$.
A canonical orientation for $V$ is given by the $(n,n)$-form
\begin{equation}\label{volume}
\tau(z)\coloneqq \frac{1}{2^n} i dz_1\wedge d\overline{z}_1 \wedge \ldots \wedge i dz_n \wedge d\overline{z}_n= dx_1 \wedge  d y_1  \wedge  d x_n \ldots \wedge dy_n,
\end{equation}
where $z_j = x_j + i y_j$.

We shall say that a $(p, p)$-form $\gamma$ is real if $\gamma=\overline{\gamma}$.
One can introduce a  natural notion  of positivity for  real $(p, p)$-forms. 

\begin{definition}\label{semipositive}
A  real $(p, p)$-form  $\gamma \in \Lambda^{p,p} V^*$ is said to be \emph{semi-positive} (resp. \emph{positive}) if, for all $\alpha_j$ of $\Lambda^{1,0}V^*$, $1\leq j\leq n-p$, 
\[
\gamma \wedge i \alpha_1\wedge \overline{\alpha}_1\wedge \ldots \wedge i \alpha_{n-p}\wedge \overline{\alpha}_{n-p}=\lambda \tau(z),
\]
where $\lambda\geq 0$ (resp. $\lambda>0$ when $\alpha_1,\ldots, \alpha_{n-p}$ are linearly indipendent).
\end{definition}

We shall focus on the case  $n = 3$ and we shall provide equivalent definitions for semi-positive real forms of type $(1,1)$ and $(2,2)$. 
For a more general discussion we refer the reader to \cite[Ch. III]{Demailly}.

\begin{proposition} \label{prop2}
Let $\alpha=\frac{i}{2} \sum_{j,k} a_{j\overline{k}} \, dz_j\wedge d\overline{z}_k$ be a real $(1,1)$-form on $V$.
Then the following are equivalent:
\begin{itemize}
\item[(i)] $\alpha$ is semi-positive (resp. positive);
\item[(ii)] the Hermitian  matrix of coefficients $(a_{j\overline{k}})$ is positive semi-definite  (resp. positive definite);
\item[(iii)]  there exist coordinates $\left(w_1,\ldots w_n\right)$  on $V$  such that 
\[
\alpha=\frac{i}{2}\sum_{k=1}^n \tilde a_{k\overline{k}} \,  d w_k \wedge d\overline{w}_k, 
\]
with  $\tilde a_{k\overline{k}} \geq 0$  (resp.  $\tilde a_{k\overline{k}}>0$), $\forall k = 1, \ldots n$.
\end{itemize}
\end{proposition}

\begin{proof} 
\mbox{} \\
(i) $\Longleftrightarrow$ (ii) follows from \cite[Ch. III, Corollary 1.7]{Demailly} and its straightforward generalization for the case of positive $(1,1)$-forms; \\
(ii) $\Longleftrightarrow$ (iii) is achieved by diagonalizing the Hermitian matrix of coefficients $(a_{j\overline{k}})$. \\
\end{proof}

The next result follows from \cite[Ch. III, Corollary 1.9, Proposition 1.11]{Demailly}.

\begin{proposition}\label{product}
If $\alpha_1, \alpha_2$ are semi-positive real $(1,1)$-forms, then  $\alpha_1 \wedge \alpha_2$  is semi-positive. 
\end{proposition}

Now, for $n =3$, we want to characterize the semi-positivity of real $(2,2)$-forms.
Let $\gamma$ be a real $(2,2)$-form on $V$. We can write 
\begin{equation}\label{gamma}
\gamma=-\frac{1}{4} \sum_{\substack{i<k \\j<l}} \gamma_{i\overline{j}k\overline{l}} dz_i\wedge d\overline{z}_j \wedge dz_k\wedge d\overline{z}_l,
\end{equation}
with respect to some coordinates $(z_1,z_2,z_3)$ on  $V$.

To $\gamma$ we can associate the real $(1,1)$-form $\beta$, given by

\[
\beta= \frac{i}{2}\sum_{m,n}\beta_{m\overline{n}}dz_m\wedge d\overline{z}_n, 
\]
where 
\begin{equation}\label{betacoef}
\beta_{m\overline{n}}\coloneqq \frac{1}{4}\sum_{i,j,k,l} \gamma_{i\overline{j}k\overline{l}} \epsilon_{ikm}\epsilon_{jln}.
\end{equation}
Here  $\epsilon_{abc}$ is the Levi-Civita symbol, with $\epsilon_{123}=1$.
Using a change of basis $dz_i=\sum_p A^p_i dw_p$, the matrix $(\beta_{m\overline{n}})$  changes by congruence via the matrix $\tilde{A}= \text{det}(A) (A^t)^{-1}$, where $A=(A^p_i)$.
Consequently, the semi-positivity of $\beta$ does not depend on the choice of coordinates on V.

Notice that the matrix $(\beta_{m\overline{n}})$ is Hermitian, since $\gamma=\overline{\gamma}$ implies $\gamma_{i\overline{j}k\overline{l}}=\overline{\gamma_{j\overline{i}l\overline{k}}}$.

\begin{proposition}\label{prop3}
Let $\gamma\neq 0$ be a real $(2,2)$-form on $V$. 
Then the following are equivalent:
\begin{itemize}
\item[(i)] $\gamma$ is semi-positive,
\item[(ii)] $\gamma\wedge \alpha >0$ for every positive real $(1,1)$-form $\alpha$, i.e. $\gamma\wedge \alpha=\lambda \tau(z)$ where $\lambda>0$,
\item[(iii)] the associated $(1,1)$-form $\beta$ is semi-positive.
\end{itemize}
\end{proposition}

\begin{proof}
\mbox{} \\
(i) $\Longleftrightarrow$ (iii) 
Let $\gamma$ be a real $(2,2)$ form on $V$. Then $\gamma$ can be written as in \ref{gamma}
with respect to a basis $(dz_1,dz_2,dz_3)$ of $\Lambda^{1,0}V^*$.
By Definition \ref{semipositive}, $\gamma$ is semi-positive if for all $\eta\in \Lambda^{1,0}V^*$ one has $\dfrac{i}{2}\gamma \wedge \eta \wedge \overline{\eta}\geq 0$. Set $\eta=\sum_m \eta_m dz_m$,
then 
\[
\dfrac{i}{2}\gamma \wedge \eta\wedge \overline{\eta} = \sum_{m,n}\beta_{m\overline{n}}\eta_m\overline{\eta}_n \tau(z),
\]
where the coefficients $\beta_{m\overline{n}}$ are defined in \ref{betacoef}.
Therefore, since $\eta$ is arbitrary, $\gamma$ is semi-positive if and only if the matrix $(\beta_{m\overline{n}})$ is positive semi-definite.\\
(i) $\implies$ (ii) Let $\alpha$ be a positive $(1,1)$-form on $V$, then there exists a basis
$\left(dz_1,dz_2,dz_3\right)$ of $\Lambda^{1,0}V^*$ such that 
\[
\alpha=\dfrac{i}{2}\sum_{k}a_{k\overline{k}}dz_k\wedge d\overline{z}_k
\] with $a_{k\overline{k}}>0$.
Let $\gamma$ be a semi-positive $(2,2)$-form on $V$.
We can write
\[
\gamma=-\frac{1}{4} \sum_{\substack{i<k \\j<l}} \gamma_{i\overline{j}k\overline{l}} dz_i\wedge d\overline{z}_j \wedge dz_k\wedge d\overline{z}_l.
\]
Then 
\[
\gamma\wedge \alpha=\sum_{r}a_{r\overline{r}}\beta_{r\overline{r}}\tau(z).
\]
Since $\gamma$ is semi-positive, by (iii) we have that $\beta_{r\overline{r}}\geq 0$ with at least one strictly positive. Therefore, since $a_{r\overline{r}}>0$, for each $r$, the claim follows.
 \\
(ii) $\implies$ (i) Let $\left(\alpha_1,\alpha_2,\alpha_3\right)$ be a basis of $\Lambda^{1,0}V^*$.
We define 
\[
\alpha_\epsilon \coloneqq \dfrac{i}{2}(\alpha_1\wedge\overline{\alpha}_1+\epsilon(\alpha_2\wedge\overline{\alpha}_2+\alpha_3\wedge\overline{\alpha}_3)).
\]
We notice that, for any $\epsilon>0$, $\alpha_\epsilon$ is a positive $(1,1)$-form.
Then, by hypothesis, $\gamma\wedge \alpha_\epsilon>0$. The claim follows by continuity since $\dfrac{i}{2}\gamma\wedge\alpha_1\wedge \overline{\alpha}_1=\lim_{\epsilon\to 0}(\gamma\wedge \alpha_\epsilon)
\geq 0$. 
\end{proof}

As shown in \cite[Theorem 1.2]{HarveyKnapp}, a real $(2,2)$-form $\gamma$  is always diagonalizable, i.e. there exist coordinates $(w_1,w_2,w_3)$ of $V$ such that 
\[
\gamma=-\frac{1}{4}   \sum_{\substack{i<k}} \gamma_{i\overline{i}k\overline{k}} dw_i\wedge d\overline{w}_i\wedge dw_k\wedge d\overline{w}_k.\]
By Proposition \ref{prop3},  $\gamma$ is semi-positive if and only if $ \gamma_{i\overline{i}k\overline{k}} \geq 0$, for every $i<k$.
In particular, the  diagonal  matrix $(\beta_{m\overline{n}})$ associated to $\gamma$ in these coordinates is positive semi-definite.   Moreover, $\gamma$ is positive if and only if 
$ \gamma_{i\overline{i}k\overline{k}}  > 0$, for every $i<k$.

\begin{remark}\cite[formula (4.8)]{Michelsohn}\label{Michelsohn}
A real $(2,2)$-form $\gamma$ on $V$ is positive if and only if $\gamma=\alpha^{2},$
where $\alpha$ is a positive $(1,1)$-form.
\end{remark}

\section{Mean convexity and intrinsic torsion of $\text{SU}(3)$-structures} \label{section3}

In this section we study the mean convex property in the context of  closed $\text{SU}(3)$-structures  and provide necessary and sufficient conditions in terms of the intrinsic torsion of the $\text{SU}(3)$-structure.

An $\text{SL}(3,\mathbb{C})$-structure on a $6$-manifold $M$ is a reduction to $\text{SL}(3,\mathbb{C})$ of
the  frame bundle  of $M$  which is given by a definite real $3$-form $\rho$, i.e. by  a stable $3$-form inducing an almost complex structure $J_{\rho}$.
We recall that   a  3-form $\rho$ on  a real $6$-dimensional space $V$  is stable if its orbit under the action of $\text{GL}(V )$ is open. 
If we fix a volume form $\nu\in\Lambda^6V^*$ and denote by \[A:\Lambda^5 V^* \to V\otimes\Lambda^6 V^* \] the canonical isomorphism induced by the wedge product 
$
\wedge:V^*\otimes \Lambda^5V^* \to \Lambda^6 V^*,
$
we can consider the map 
$$
K_{\rho} : V  \to V\otimes\Lambda^6 V^*,  \quad v  \mapsto  A((\iota_v \rho)\wedge \rho).
$$
A $3$-form $\rho$ on $V$ is stable if and only if   $\lambda(\rho) =\frac{1}{6}\, \text{Tr}(K^2_{\rho}) \neq 0$ (see \cite{Hitchin,Reichel} for further details).
When $\lambda(\rho)<0$, the $3$-form $\rho$ induces an almost complex structure 
\[
J_{\rho} := -\frac{1}{\sqrt{-\lambda(\rho)}}K_{\rho}
\]
and we shall say that $\rho$ is \emph{definite}.
A simple computation shows that $J_{\rho}$ does not change if $\rho$ is rescaled by a non-zero real constant, i.e., $J_{\rho}=J_{s \rho}$ for every $s\in \mathbb{R}-\{0\}$.
Moreover,  defining $\hat \rho \coloneqq J_{\rho}\rho$, we have that $ \rho + i \hat \rho$ is a complex $(3,0)$-form with respect to $J_{\rho}$.

We shall say that an  $\text{SL}(3,\mathbb{C})$-structure  $\rho$  is closed if $d \rho =0$. According to \cite{Donaldson},  $d\hat \rho$ is a  real $(2,2)$-form  and so we can introduce the following

\begin{definition}\label{meanconvex}
Let $\rho$ be a closed $\text{SL}(3,\mathbb{C})$-structure on $M$.
We shall say that $\rho$ is \emph{mean convex} (resp. strictly mean convex) if $d\hat \rho$, pointwise, is a non-zero semi-positive (resp. positive) $(2,2)$-form.
\end{definition}

Given an  $\text{SL}(3,\mathbb{C})$-structure  $\rho$ and
a  non-degenerate positive $(1,1)$-form  $\omega$  on on a $6$-manifold  $M$ such  that  $\rho \wedge \hat \rho=\frac{2}{3}\omega^3$, then the pair $(\omega, \Psi)$, where $\Psi=\rho+iJ_{\rho}\hat{\rho}$, defines  an $\text{SU}(3)$-structure and the associated almost  $J_{\rho}$-Hermitian metric $g$ is given by  $g (\cdot , \cdot) \coloneqq \omega(\cdot,J_{\rho}\cdot)$.
Since $\Psi$ is completely determined by its real part $\rho$, we shall denote   an $\text{SU}(3)$-structure  simply by  the pair $(\omega,\rho)$.

In this case, at any point $p\in M$, one can always find a coframe $\left(f^1,\ldots,f^6\right)$, called \emph{adapted basis} for the $\text{SU}(3)$-structure $(\omega, \rho)$, such that
\begin{equation}\label{frameadattato}
\omega=f^{12}+f^{34}+f^{56}, \quad \rho=f^{135}-f^{146}-f^{236}-f^{245}.
\end{equation}
Here $f^{ij\cdots k}$ stands for the wedge product $f^i\wedge f^j\wedge \cdots \wedge f^k$.

We shall say that the $\text{SU}(3)$-structure $(\omega, \rho)$ is closed if $d \rho=0$ and  in a similar way we can introduce   the following

\begin{definition}
A  closed $\text{SU}(3)$-structure  $(\omega,\rho)$ on a $6$-manifold  $M$  is (strictly) mean convex if the $\text{SL}(3,\mathbb{C})$-structure $\rho$ is (strictly) mean convex.
\end{definition}

The intrinsic torsion of  the  $\text{SU}(3)$-structure $(\omega, \rho)$   can be identified with the pair $(\nabla \omega, \nabla \Psi)$, where $\nabla$ is the Levi-Civita connection of $g$,  and it is a section of the vector bundle 
$T^*M\otimes \mathfrak{su}(3)^{\perp}$,  where $\mathfrak{su}(3)^{\perp} \subset \mathfrak{so}(6)$ is the orthogonal complement of $\mathfrak{su}(3)$ with respect to the Killing Cartan form $\mathcal{B}$ of $\mathfrak{so}(6)$.
Moreover, by  \cite[Theorem 1.1]{ChiossiSalamon}  the intrinsic torsion of $(\omega, \rho)$   is completely determined by  $d\omega$, $d \rho$ and  $d \hat \rho$.
Indeed,  there exist unique differential forms $\nu_0, \pi_0 \in C^{\infty}(M)$, $\nu_1,\pi_1 \in \Lambda^1(M)$, $\nu_2,\pi_2\in 
[\Lambda^{1,1}_0 M ], \nu_3 \in \llbracket \Lambda^{2,1}_0 M \rrbracket $ such that
\begin{equation}\label{torsionforms}
\begin{aligned}
d\omega &=-\frac{3}{2} \nu_0 \, \rho + \frac{3}{2}\pi_0  \, \hat \rho + \nu_1 \wedge \omega + \nu_3, \\
d\rho &=\pi_0 \, \omega^2 + \pi_1 \wedge \rho - \pi_2\wedge \omega, \\
d\hat \rho  &=\nu_0\,\omega^2-\nu_2\wedge \omega +J\pi_1\wedge \rho,
\end{aligned}
\end{equation}
where $[\Lambda^{1,1}_0 M ]\coloneqq \{ \alpha \in [\Lambda^{1,1} M ] ~ | ~ \alpha\wedge \omega^2=0 \}$ is the space of primitive real $(1,1)$-forms and $\llbracket \Lambda^{2,1}_0 M \rrbracket \coloneqq \{ \eta \in \llbracket \Lambda^{2,1} M \rrbracket ~ | ~ \eta \wedge \omega=0 \}$ is the space of primitive real $(2,1)+(1,2)$-forms.
The forms $\nu_i,  \pi_j$ are called \emph{torsion forms} of the $\text{SU}(3)$-structure and they completely determine its intrinsic torsion, which vanishes if and only if all the torsion forms vanish identically.

If  $\rho$ is closed,   as a consequence of  \ref{torsionforms}, we have  $d\hat \rho  = \theta \wedge \omega,$ where $\theta$  is 
the $(1,1)$-form defined by
$\theta \coloneqq\nu_0\,\omega -\nu_2$. 

We recall that, given a  real $(1,1)$-form $\alpha$,  the   trace   $\normalfont \operatorname{Tr} (\alpha)$  of   $\alpha$   is given by 
$3\alpha\wedge\omega^{2}=\normalfont \operatorname{Tr} (\alpha)\omega^3.$
Then, in terms of $\nu_0$ and  the $(1,1)$-form $\theta$, we can prove the following 

\begin{proposition}\label{prop5}
Let $(\omega,\rho)$  be a closed  $\text{\normalfont SU}(3)$-structure  on $M$.
Then
\begin{itemize}
\item[(i)] if $(\omega,\rho)$ is mean convex, then  the torsion form $\nu_0$ is strictly positive and the $(1,1)$-form $\theta$ is not  negative (semi-)definite. Moreover,  its trace $\normalfont \operatorname{Tr}(\theta)$   is strictly positive,
\item[(ii)] if $\theta$ is semi-positive, then the $\text{\normalfont SU}(3)$-structure is mean convex.
\end{itemize}
\end{proposition}

\begin{proof}
Let us assume that $(\omega,\rho)$ is a mean convex closed $\text{SU}(3)$-structure on $M$.
By \ref{torsionforms} we have  $d\hat \rho=\theta \wedge \omega$. 
Now, Proposition \ref{prop3} implies  $d\hat \rho \wedge \alpha>0 $ for every positive real $(1,1)$-form $\alpha$.
Then (i) follows by choosing $\alpha=\omega$; indeed $d\hat \rho \wedge \omega=\nu_0 \omega^3$,  since $\nu_2\in [\Lambda^{1,1}_0 M]$. In particular $\text{Tr}(\theta)=3 \nu_0>0$.
(ii) follows from Proposition \ref{product}.
\end{proof}

A closed $\text{SU}(3)$-structure $(\omega,\rho)$ is  called \emph{half-flat} if $d \omega^2 =0$ and we shall refer to it simply  as a  half-flat structure.  Half-flat structures are strictly related to torsion free $\text{G}_2$-structures.  
We recall that  a $\text{G}_2$-structure on a $7$-manifold $N$ is   characterized by the existence of a $3$-form $\varphi$ inducing a Riemannian metric $g_{\varphi}$ and a volume form $dV_{\varphi}$ given by 
\[
g_{\varphi}(X,Y)dV_{\varphi}=\frac{1}{6}\iota_X \varphi \wedge \iota_Y \varphi \wedge \varphi, \quad X, Y \in \Gamma (TM).
\]
 By \cite{FG},  the  $\text{G}_2$-structure $\varphi$  is \emph{torsion free}, i.e. $\varphi$ is parallel with respect to the Levi-Civita connection of $g_{\varphi}$, if and only if $\varphi$  is closed and co-closed,  or equivalently if the holonomy group $\text{Hol}(g_\varphi)$ is contained in $ \text{G}_2$.
A torsion free $\text{G}_2$-structure $\varphi$  on $N$ induces on each oriented  hypersurface $\iota :M\hookrightarrow N$   a natural half-flat structure $(\omega,\rho)$ given by
\[
\rho=\iota^*\varphi, \quad \omega^2=2  \, \iota^*(*_{\varphi} \varphi).
\]
Conversely, in \cite{Hitchin2}, the   so-called Hitchin flow equations
\begin{equation}\label{HitchinFlow}
\begin{cases}
\frac{\partial}{\partial t} \rho(t)=d\omega(t), \\
\frac{\partial}{\partial t} \omega(t)\wedge \omega(t)=-d\hat{\rho}(t),
\end{cases}
\end{equation}
have been introduced, proving that every compact real analytic half-flat manifold $(M, \omega,\rho)$ can be embedded isometrically as a hypersurface in a $7$-manifold $N$ with  a torsion free $\text{G}_2$-structure. 
Moreover, the intrinsic torsion of the half-flat structure can be identified with the second fundamental form $B\in \Gamma( S^2 T^*M)$ of $M$  with respect to a fixed unit normal vector field $\xi$.
As in \cite{Donaldson}, with respect to  $J_{\rho}$, we can  write  $B=B_{1,1}+B_C$, where $B_{1,1}$ is the real part of a Hermitian form and $B_C$ is the real part of a complex quadratic form. If we denote by $\beta_{1,1}=B_{1,1}(J_{\rho}\cdot,\cdot)$ the corresponding $(1,1)$-form on $M$, we have $\beta_{1,1}\wedge\omega=\frac{1}{2}d\hat{\rho}$, from which it follows that, if $(\omega,\rho)$ is mean convex, then the mean
curvature $\mu$ given explicitly by $\frac{1}{4}\mu \rho\wedge \hat{\rho} =\frac{1}{2} d\hat{\rho}\wedge\omega$ is positive with respect to the normal direction (for more details see  \cite[Prop.  1]{Donaldson}).
Moreover, since the wedge product with $\omega$ defines an injective map on $2$-forms, comparing this with \ref{torsionforms} yields $\theta=2\beta_{1,1}$. Then, by Proposition \ref{prop5}, if $B_{1,1}$ defines a positive semi-definite Hermitian product, then the half-flat structure $(\omega,\rho)$ is mean convex. 

\smallskip

Special types of half-flat structures $(\omega,\rho)$ are called \emph{coupled}, when $d\omega=-\frac{3}{2}\nu_0 \,\rho$, and 
 \emph{double}, when 
$
d\hat \rho=\nu_0 \,\omega^2.
$

Notice that, by Proposition \ref{prop5}, double structures $(\omega, \rho)$  are trivially mean convex as long as $\nu_0>0$. However, it is straightforward to check that, if $(\omega,\rho)$ is a double structure  such that  $
\nu_0<0$, then $(\omega, -\rho )$ is  mean convex.

In \cite[Theorem 4.11]{ChiossiSwann},   a classification of  $6$-dimensional  nilpotent Lie algebras  endowed with a double structure was given. Other examples of double structures on $S^3\times S^3$ were found in  \cite{MadsenSalamon,schulte}. 

For a general Lie algebra we can show the following

\begin{proposition} 
If a Lie algebra  $\mathfrak{g}$ has a  closed strictly mean convex $\normalfont \text{SL}(3,\mathbb{C})$-structure, then $\mathfrak{g}$  admits a  double  structure.
\end{proposition}

\begin{proof}
Let $\rho$ be a  closed strictly mean convex $\text{SL}(3,\mathbb{C})$-structure on $\mathfrak{g}$ and denote $\hat{\rho}= J_{\rho}\rho$ as usual. Then $d\hat{\rho}$ is a positive $(2,2)$-form and, as shown in \cite{Michelsohn} (see Remark \ref{Michelsohn}), there exists a positive $(1,1)$-form $\alpha$ such that $d\hat{\rho}=\alpha^2$.
Moreover, since $\alpha$ is positive with respect to  $J_{\rho}$, $\alpha^3$ is a positive multiple of the volume form $\rho\wedge \hat{\rho}$.
Since $J_{\rho}$ does not change for  a non-zero rescaling of $\rho$, this implies that  there exists $b\neq 0$  such that  $(b  \rho ,\alpha) $  is a double structure on $\mathfrak{g}$.
\end{proof}

As a consequence,  the classification of  nilpotent Lie algebras  admitting closed strictly mean convex $\text{SL}(3,\mathbb{C})$-structures  reduces to Theorem 4.11 in \cite{ChiossiSwann}.
Therefore, in the next two sections we weaken the condition asking for the existence of closed (non-strictly) mean convex  $\text{SL}(3,\mathbb{C})$-structures.

\section{Proof of Theorem A}\label{section4}

We recall that a \emph{nilmanifold}  $M= \Gamma \backslash G$  is a  compact quotient  of  a  connected, simply connected, nilpotent Lie group  $G$  by a  lattice  $\Gamma$.  We shall say that  an $ \text{SL}(3,\mathbb{C})$-structure $\rho$  (resp.  $\text{SU}(3)$-structure $(\omega,\rho)$)  is \emph{invariant}   if it is  induced by  a left-invariant one  on the nilpotent Lie group $G$.  Therefore, the study of  these types of structure is equivalent to the study   of $ \text{SL}(3,\mathbb{C})$-structures (resp.  $\text{SU}(3)$-structures)   on  the Lie algebra $\mathfrak{g}$ of $G$ and we can work at the level of nilpotent Lie algebras.

Six-dimensional nilpotent Lie algebras have been  classified in \cite{Goze, Magnin}. Up to isomorphism,
they are $34$, including the abelian algebra (see Table \ref{table1} for the list).
Using this classification we can prove Theorem A.

\begin{proof}[Proof of Theorem A]
Let $\mathfrak{g}$ be  the  Lie algebra of $G$. Every invariant  $\normalfont \text{SL}(3,\mathbb{C})$-structure on $M$
is determined by  an $\normalfont \text{SL}(3,\mathbb{C})$-structure on $\mathfrak{g}$ and vice versa.
First note that the possibility that $\mathfrak{g}$ is abelian is precluded by Definition \ref{meanconvex}.
Then, in order to prove the first part of the  theorem, we first  show the non existence result for the five Lie algebras $\mathfrak{g}_{1}$, $\mathfrak{g}_{2}$, $\mathfrak{g}_{4}$, $\mathfrak{g}_{9}$ and $\mathfrak{g}_{12}$.
For any of these  Lie algebras, let us consider a generic  closed 3-form
\[
\rho=\sum_{i<j<k}p_{ijk}\, e^{ijk}, \quad p_{ijk}\in \mathbb{R}.
\]
Let us assume that $\rho$ is definite, i.e.\ stable with $\lambda(\rho)<0$. 
Then $\rho$ induces an almost complex structure $J_{\rho}$ and we may ask if the induced $(2,2)$-form $d\hat \rho$ is semi-positive. 
Notice that the $1$-forms $\zeta^k=e^k-iJ_{\rho}e^k$, for $k=1,\ldots,6$, generate the space  $\Lambda^{1,0}\mathfrak{g}_i^*$ of $(1,0)$-forms with respect to $J_{\rho}$ on   $\frak g_i$, $i = 1,2,4,9,12.$ Here we are using the convention 
$J_{\rho}\alpha(v)=\alpha(J_{\rho}v)$ for any $\alpha\in \mathfrak{g}^*$, $v\in\mathfrak{g}$. 
So, for any closed definite $3$-form $\rho$, we extract a basis $(\xi^1,\xi^2,\xi^3)$ for $\Lambda^{1,0}\mathfrak{g}_i^*$, where $\xi^j=\zeta^{k_j}$ for some $k_j\in \{1,\ldots,6\}$ and $j=1,2,3$. Then, $(\xi^1,\xi^2,\xi^3, \overline{\xi}^1, \overline{\xi}^2, \overline{\xi}^3)$ is a complex basis for $\mathfrak{g}_i^*\otimes \mathbb{C}$ and we can write $d\hat{\rho}$ in this new basis as
\[
d\hat \rho=-\frac{1}{4}\sum_{\substack{i<k\\j<l}}\gamma_{i\overline{j}k\overline{l}} \, \xi^i\overline{\xi}^j\xi^k\overline{\xi}^l,
\]
for some $\gamma_{i\overline{j}k\overline{l}}\in \mathbb{C}$.
We note that the real one-forms 
\[
e^{k_j}=\frac{1}{2}(\xi^j+\overline{\xi}^j), \quad J_{\rho}(e^{k_j})=\frac{i}{2}(\xi^j-\overline{\xi}^j), \quad j=1,2,3,
\]
define a new real basis for $\mathfrak{g}_i^*.$ 
Now, following Section \ref{section2}, we consider the real $(1,1)$-form $\beta$  associated to $d\hat \rho$, given explicitly by
\begin{equation} \label{exprbeta}
 \beta=\frac{i}{2}\sum_{m,n} \beta_{m\overline{n}} \, \xi^m\overline{\xi}^n, \quad  \beta_{m\overline{n}}=\frac{1}{4}\sum_{i,j,k,l} \gamma_{i\overline{j}k\overline{l}}\epsilon_{ikm}\epsilon_{jln},
\end{equation}
and we  compute  the expression of $\beta_{m\overline{n}}$ in terms of $p_{ijk}$.
Therefore, $d\hat{\rho}$ is semi-positive (non-zero) if and only if the Hermitian matrix $(\beta_{m\overline{n}})$ is positive semi-definite, which occurs if and only if 
\begin{equation}\label{criterio}
\begin{dcases}
\beta_{k\overline{k}}\geq 0, &k=1,2,3, \\
\beta_{r\overline{r}}\beta_{k\overline{k}}-\lvert \beta_{r\overline{k}}\rvert^2\geq 0,  &r<k, \, \,r,k=1,2,3, \\
\det(\beta_{m\overline{n}})\geq 0,
\end{dcases}
\end{equation}
with $(\beta_{m\overline{n}})$ different from the zero matrix.

Then it can be shown that, for every closed $3$-form $\rho$   such that  $\lambda (\rho)<0$, the system \ref{criterio}  in the variables $p_{ijk}$    has no solutions.

Let us see this explicitly for  $\mathfrak{g}_i$, $i=1,2$. 
By a direct computation,  for the generic closed $3$-form  $\rho$ on $\mathfrak{g}_1$ we have
$$
\lambda(\rho)=\left[(p_{145}+2p_{235})p_{146}+p_{145}p_{236}+p_{245}^2\right]^2 +4p_{146}p_{236}\left(p_{126}-p_{145}p_{235}+p_{135}p_{245}\right)
$$
and, for the generic closed $3$-form  $\rho$ on $\mathfrak{g}_2$,  we get
$$
\lambda(\rho)=\left(p_{245}^2+p_{145}p_{236}+2p_{146}p_{235}\right)^2+4p_{146}p_{236}\left(-p_{145}p_{235}+p_{135}p_{245}+p_{125}p_{146} \right).
$$
Notice that, if at least one between $p_{146}$ and $p_{236}$ is equal to zero, then $\lambda(\rho)\geq 0$. So let us assume that both $p_{146}, \, p_{236}$ are non-zero. Then $(e^1,J_{\rho} e^1,e^2,J_{\rho} e^2,e^5,J_{\rho} e^5)$ defines a  basis of $\mathfrak{g}_i^*$, for $i=1,2$, hence  ($\xi^1=e^1-iJ_{\rho} e^1,\xi^2=e^2-iJ_{\rho} e^2,\xi^3=e^5-iJ_{\rho} e^5$)  is  a basis of $(1,0)$-forms on $\mathfrak{g}_i$, $i=1,2$. 
By a direct computation, it can be shown that in these cases the matrix coefficient $\beta_{1\overline{1}}$ vanishes and so
$ \beta_{1\overline{1}}\beta_{3\overline{3}}-\lvert \beta_{1\overline{3}}\rvert^2=-\lvert \beta_{1\overline{3}}\rvert^2 \leq 0$, but $\beta_{1\overline{3}}=0$ implies $\lambda(\rho)=0$ which is a contradiction.

By a very similar discussion, we may discard cases $\mathfrak{g}_4$,  $\mathfrak{g}_9$ and  $\mathfrak{g}_{12}$ as well.
In order to prove the second part of the theorem, we construct an explicit  mean convex  closed $\text{SU}(3)$-structure $(\omega, \rho)$  on the remaining nilpotent Lie algebras  (see Table \ref{table2}).
\end{proof}

\section{Proof of Theorem B}\label{section5} 

In \cite{Conti}, a  classification up to isomorphism  of  $6$-dimensional real nilpotent Lie algebras admitting half-flat structures was given.  The non-abelian ones are twenty three  and they   are listed in Table \ref{table1}. 
So, in order to classify nilpotent  Lie algebras admitting  a mean convex  half-flat structure, we restrict our attention to this list.
An explicit example of mean convex  half-flat structure on  $\mathfrak{g}_i$, $i=6,7,8,10,13,15,16,22, 24,$ $25, 28,$ $29,30,31,32,33$,  is already  given in Table \ref{table2}.
Therefore,  we only need to prove non-existence of  mean convex  half-flat structures on the remaining Lie algebras $\mathfrak{g}_i$, $i=4,9,11,12,14,21,27$.
By Theorem \hyperref[theoremA]{A}, we may immediately  exclude   the Lie algebras $\mathfrak{g}_i$, $i=4,9,12$, since mean convex half-flat structures   are in particular mean convex  closed $\text{SL}(3,\mathbb{C})$-structures.

For the remaining Lie algebras $\mathfrak{g}_i$, $i=11,14,21,27$, whose first Betti number is $3$ or $4$,
we first collect some necessary conditions to the existence of mean convex closed $\text{SU}(3)$-structures $(\omega, \rho)$ in terms of a filtration of $J_{\rho}$-invariant subspaces $U_i$  of $\mathfrak{g}^*$, and then, by working in an $\text{SU}(3)$-adapted basis, we exhibit further obstructions.

Let us start by defining the  filtration $\{ U_i \}$ as in \cite{ChiossiSwann}.  Let $(\omega,\rho)$ be an $\text{SU}(3)$-structure on  a $6$-dimensional nilpotent Lie algebra $\frak g$  and let $(g,J_{\rho})$ be the induced almost Hermitian structure on $\mathfrak{g}$.  By nilpotency there exists a basis $\left(\alpha^1,\ldots,\alpha^6\right)$ of $\mathfrak{g}^*$   such that, if we denote $V_j\coloneqq \left< \alpha^1,\ldots,\alpha^j\right>$, then $dV_j\subset \Lambda^2V_{j-1}$ and, by construction, $0\subset V_1\subset\ldots \subset V_5\subset V_6=\mathfrak{g}^*$.
We notice that the basis $(e^i)$ whose corresponding structure equations are given in Table \ref{table1} satisfies the previous conditions and $V_i=\ker d$ when $b_1(\mathfrak{g})=i$. In the following, we consider $V_i= \left< e^1,\ldots,e^i\right>$.
As in \cite{ChiossiSwann}, let $U_j\coloneqq V_j\cap J_{\rho} V_j$ be the maximal $J_{\rho}$-invariant subspace of $V_j$ for each $j$. Then, since $J_{\rho}$ is an automorphism  of  the vector space $\mathfrak{g}$,  a simple dimensional computation shows that $\dim_{\mathbb{R}}U_2$, $\dim_{\mathbb{R}}U_3 \in \{0,2\}$, $\dim_{\mathbb{R}}U_4\in \{2,4\}$ and $\dim_{\mathbb{R}}U_5=4$.
Note that the filtration $\{ U_i \}$ depends on $V_i$ and the  almost complex structure $J_{\rho}$.

We can  prove the following

\begin{lemma}\label{U3=2,U2=2} 
Let $\rho$ be  a mean convex closed $\normalfont \text{SL}(3,\mathbb{C})$-structure on a nilpotent Lie algebra  $\mathfrak{g}$. If $\mathfrak{g}$ is isomorphic to 
\[\mathfrak{g}_{11}=(0,0,0,e^{12},e^{14},e^{15}+e^{23}+e^{24}) \quad \text{or} \quad \mathfrak{g}_{14}=(0,0,0,e^{12},e^{13},e^{14}+e^{35}), 
\]
then $U_3=U_4$. 
If $\mathfrak{g}$ is isomorphic to 
\[
\mathfrak{g}_{21}=(0,0,0,e^{12},e^{13},e^{14}+e^{23})  \quad \text{or} \quad   \mathfrak{g}_{27}=(0,0,0,0,e^{12},e^{14}+e^{25}),
\]
then $\dim_{\mathbb{R}}U_2=2$, or equivalently  $\left < e^1, e^2  \right > $  is  $J_{\rho}$-invariant.
Moreover, on $\mathfrak{g}_{21}$, up to isomorphism, we also have $\dim_{\mathbb{R}}U_4=4$.
\end{lemma}

\begin{proof}
On each  Lie algebra $\mathfrak{g}_i$, $i =11, 14, 21, 27$,  we consider the  generic  closed    3-form
\[
\rho=\sum_{i<j<k}p_{ijk}\, e^{ijk}, \quad p_{ijk}\in \mathbb{R}
\] 
and we impose  $ \lambda (\rho) <0$ and the mean convex condition. 
First, by a  direct computation on each Lie algebra, we determine the expression of $\lambda (\rho)$ in terms of the coefficients $p_{ijk}$ and a basis of $(1,0)$-forms with respect to $J_{\rho}$.  Then we exclude  the  cases  where  either  $\lambda (\rho)  \geq 0$  or the matrix $(\beta_{m\overline{n}})$ associated  to  $d\hat{\rho}$ is not  positive semi-definite. As in the proof of Theorem \hyperref[theoremA]{A} we first  extract a basis of $(1,0)$-forms  from the set of generators $\{ \zeta^ i  \}$ and we use \ref{exprbeta} to compute  $(\beta_{m\overline{n}})$ in terms of $p_{ijk}$.  We shall give  all the details for the Lie algebra $\mathfrak{g}_{11}$. For the other cases the computations are similar and we only report the necessary conditions on $p_{ijk}$.   The generic closed $3$-form $\rho$  on  the Lie algebra $\mathfrak{g}_{11}$  has
\begin{align*}
\lambda(\rho)=& ( p_{126}p_{236}-p_{126}p_{146}-p_{135}p_{246}+p_{145}p_{236}+p_{146}p_{235}-p_{146}p_{245} +p_{234}p_{246} \\ & -p_{235}p_{245})^2+ 4p_{246}(p_{123}p_{236}p_{246}-p_{123}p_{246}^2-p_{124}p_{236}^2 +p_{124}p_{236}p_{246} \\
&+2p_{125}p_{146}p_{236} -p_{125}p_{146}p_{246}+p_{125}p_{235}p_{236}-p_{125}p_{235}p_{246}-p_{134}p_{235}p_{246} \\
& +p_{134}p_{236}p_{245}-p_{125}p_{146}p_{246}+p_{135}p_{234}p_{246}-p_{135}p_{235}p_{245}+p_{145}p_{146}p_{235} \\
&+p_{145}p_{235}^2-p_{145}p_{234}p_{236}) +4p_{146}p_{236}(-p_{125}p_{236}+p_{135}p_{235}-p_{145}p_{235}).
 \end{align*}
Then we  have  the following possibilities: 
\begin{itemize}
\item[(a)] $p_{246}\neq 0, p_{246}\neq p_{236}$. Then $\left(e^1-iJ_{\rho}e^1,e^2-iJ_{\rho}e^2,e^3-iJ_{\rho}e^3\right)$ is a basis for $\Lambda^{1,0}\mathfrak{g}_{11}^*$, but $(\beta_{m\overline{n}})$ being positive semi-definite implies $\lambda (\rho) =0$, a contradiction.
\item[(b)]  $p_{246}=0, p_{236}\neq 0, p_{146}\neq 0$. Taking $\left(e^1-iJ_{\rho}e^1,e^2-iJ_{\rho}e^2,e^5-iJ_{\rho}e^5\right)$ as a basis for $\Lambda^{1,0}\mathfrak{g}_{11}^*$, again we find that $(\beta_{m\overline{n}})$ being positive semi-definite implies $\lambda (\rho) =0$.
\item[(c)]  $p_{246}=p_{236}=0,$ or $p_{246}=p_{146}=0$, but  then $\lambda (\rho) \geq 0$.
\item[(d)]  $p_{236}=p_{246}\neq 0$. In particular this implies  that $V_2=\left<e^1,e^2\right>$ is $J_{\rho}$-invariant, i.e., 
$\dim_{\mathbb{R}}U_2=2$. Notice also that, since $J_{\rho}  e^3 (e_6)=0$  if and only if $p_{236}=0$,
  we also have that  $V_4=\left<e^1,e^2,e^3,e^4\right>$ is not $J_{\rho}$-invariant, hence $U_2=U_3=U_4$.
\end{itemize} 

By a very similar discussion, one can show that a generic  mean convex closed $\normalfont \text{SL}(3,\mathbb{C})$-structure   $\rho$ on $\mathfrak{g}_{14}$ must have $p_{245}=0$  and $p_{356}\neq 0$. In particular, since $J_{\rho} e^1, J_{\rho} e^3 \in \left<e^1,e^3\right>$, we have $\dim_{\mathbb{R}} U_3=2$.
Moreover, $J_{\rho}  e^2 (e_6)\neq 0$, hence $\dim_{\mathbb{R}}U_2=0$ and $U_3=U_4$.

Analogously, every  mean convex closed $\normalfont \text{SL}(3,\mathbb{C})$-structure  $\rho$ on $\mathfrak{g}_{21}$ must have $p_{345}=0$.
This implies that $V_2$ and $V_4$ are $J_{\rho}$-invariant, so that $\dim_{\mathbb{R}} U_2=2$, $\dim_{\mathbb{R}}U_4=4$ and $U_2=U_3$.

Finally, a  mean convex closed $\normalfont \text{SL}(3,\mathbb{C})$-structure $\rho$ on $\mathfrak{g}_{27}$ must have $p_{345}=0$. In particular this implies that $V_2$ is $J_{\rho}$-invariant so that $U_2=U_3$.
\end{proof}

Now we can prove Theorem B.

\begin{proof}[Proof of Theorem B]
Starting from the classification of  half-flat nilpotent Lie algebras given in \cite{Conti},
we divide the discussion depending on the first Betti number $b_1$ of $\mathfrak{g}$.

When $b_1(\mathfrak{g})=2$, the claim follows directly by Theorem \hyperref[theoremA]{A}. In particular we have seen that $\mathfrak{g}_4$ cannot admit mean convex closed $\text{SL}(3,\mathbb{C})$-structures and, for the remaining Lie algebras $\mathfrak{g}_6$, $\mathfrak{g}_7$ and $\mathfrak{g}_8$ from Table \ref{table1}, we provide an explicit example in Table \ref{table2} on the respective Lie algebras.

Analogously, when $b_1(\mathfrak{g})=3$, an explicit example of mean convex  half-flat structure on $\frak g_i$, $i=10,13,15,16,22,24$, is given in Table \ref{table2}.
By Theorem \hyperref[theoremA]{A}, we may
exclude the existence of mean convex  half-flat structures  on $\mathfrak{g}_9$ and $\mathfrak{g}_{12}$.
For the remaining  Lie algebras $\mathfrak{g}_{i}$, $i=11,14,21$, let $(\omega,\rho)$ be a mean convex  half-flat  structure on $\mathfrak{g}_i$. Then,  by  Lemma \ref{U3=2,U2=2}, with respect to the fixed nilpotent filtration $V_i = \left < e^1, \ldots, e^i \right >$, we may assume $\dim_{\mathbb{R}}U_3=2$. Using this and the information on $U_4$ we collected in Lemma \ref{U3=2,U2=2}, we  shall show that on the three Lie algebras there exists an adapted basis  $(f^i)$ with dual basis $(f_i)$ such that $df^1=df^2=0$ and $f_6\in \xi(\mathfrak{g}_i)$, where by  $ \xi(\mathfrak{g}_i)$ we denote   the center of $\frak g_i$.

To see this, let us consider the case of $\mathfrak{g}_{21}$, first. Then we may assume $\dim_{\mathbb{R}}U_4=4$. This occurs if and only if $V_4=J_{\rho} V_4$. In particular, we may choose a $g$-orthonormal basis $\left(f^1,f^2\right)$ of $U_3$ such that $J_{\rho} f^1=-f^2$, take $f^3,f^4\in U_3^{\perp}\cap U_4$ of unit norm such that $J_{\rho} f^3=-f^4$, and complete it to a basis for $\mathfrak{g}_{21}^*$ by choosing $f^5\in U_4^{\perp}\cap V_5$ and $f^6\in  U_4^{\perp} \cap J_{\rho}V_5$ of unit norm such that $J_{\rho} f^5=-f^6$. Then, by construction, $\left(f^1,\ldots,f^6\right)$ is an adapted basis for the $\text{SU}(3)$-structure $(\omega,\rho)$. In particular, since $V_5=\left<f^1,f^2,f^3,f^4,f^5\right>$, the inclusion $dV_j\subset \Lambda^2(V_{j-1})$ implies $f_6\in \xi(\mathfrak{g}_{21})$. Therefore, since $f^1,f^2\in V_3=\ker d$, we have $df^1=df^2=0$.

Now we consider $\mathfrak{g}_{11}$ and $\mathfrak{g}_{14}$. By Lemma \ref{U3=2,U2=2}, we can assume $\dim_{\mathbb{R}}U_4=2$ for  both  Lie algebras.
As shown in \cite{ChiossiSwann}, since $U_4,V_3\subset V_4$, we have $\dim_{\mathbb{R}}(U_4\cap V_3)\geq 1$ and we may take $\left(f^1,f^2\right)$ to be a unitary basis of $U_4$ with $f^1\in V_3$. Then, since $U_3\subset V_3=\ker d$, we may suppose $df^1=df^2=0$.
Analogously, since $\dim_{\mathbb{R}}(V_4\cap J_{\rho}V_5)\geq 3$ and
$U_5\cap V_4=V_5 \cap J_{\rho}V_5\cap V_4= V_4\cap J_{\rho}V_5$, then $\dim_{\mathbb{R}}(U_5\cap V_4)\geq 3$, from which 
$\dim_{\mathbb{R}}(U_5\cap V_4\cap U_4^{\perp})\geq 1$ follows. Then we may take $\left(f^3,f^4\right)$ to be a unitary basis of $U_4^{\perp}\cap U_5$ with $f^3\in V_4$. 
Finally, since  $\dim_{\mathbb{R}}(U_5^{\perp}\cap V_5)\geq 1$, we may take a unitary basis $\left(f^5,f^6\right)$ of $U_5^{\perp}$ with $f^5\in V_5$. By construction, $\left(f^1,f^2,\ldots,f^6\right)$ is an adapted basis for $(\omega,\rho)$. In particular, since $U_5\subset V_5$, we also have $V_5=\left<f^1,f^2,f^3,f^4,f^5\right>$, which implies $f_6\in \xi(\mathfrak{g_i})$. for $i=11,14$. This proves our claim. 

Now, we shall show that   the three Lie algebras $\mathfrak{g}_{i}$, $i=11,14,21$,  do not admit any  mean convex half-flat structures.
By contradiction, let us suppose there exists a nilpotent Lie algebra $\mathfrak{g}$ endowed with a mean convex half-flat structure $(\omega,\rho)$ which is isomorphic to $\mathfrak{g}_{11}$, $\mathfrak{g}_{14}$ or $\mathfrak{g}_{21}$.  By the previous discussion, without loss of generality, we may assume  that there exists an adapted basis $(f^i)$, i.e. satisfying
$$
\omega=f^{12}+f^{34}+f^{56},\quad \rho=f^{135}-f^{146}-f^{236}-f^{245},  \quad \hat{\rho}=f^{136}+f^{145}+f^{235}-f^{246},
$$
and   such that  $df^1=df^2=0$, $f_6\in \xi(\mathfrak{g})$.
In particular, $\frak g$ has structure equations
\[
df^1=df^2=0, \quad \displaystyle df^k=-\sum_{\substack{ i<j\\ i,j=1}}^5 c_{ij}^k f^{ij}, \quad k=3,4,5,6.
\]
By imposing the unimodularity  of $\frak g$,   i.e.  $ \sum_{j} c_{ij}^j=0$, for all $i=1,\ldots,6$,  and  that $(\omega, \rho)$ is  half-flat, 
 we can show by a direct computation that, if  $c_{34}^5\neq 0$, then  the Jacobi identities   $d^2 f^i =0$, $i = 3, \ldots, 6$, are equivalent to the conditions
\[
c_{15}^4=c_{25}^4=c_{25}^3=c_{15}^6=c_{13}^4=c_{14}^4=c_{13}^3=c_{23}^3=c_{24}^3=0,   
\]
which imply $b_1(\mathfrak{g})\geq 4$,  so we can exclude this case.
Then we must have  $c_{34}^5=0$. Let  us assume $c_{12}^6\neq 0$.
Again a straightforward computation shows that
$d^2f^6=0$ implies 
\[
c_{25}^3=c_{25}^4=c_{15}^4=0, \quad c_{13}^3=-c_{14}^4, \quad c_{23}^3=-c_{13}^4-c_{15}^6.
\]
Now let us look at the mean convex condition. 
Since we are working in the adapted basis $(f^i)$, using \ref{exprbeta} we obtain that the matrix $(\beta_{m\overline{n}})$ associated to $d\hat{\rho}$, with respect to the  basis  $(\xi^1=f^1+if^2,\xi^2=f^3+if^4,\xi^3=f^5+if^6)$,  is given  by
\[
\begin{pmatrix}
0 & 0 & 0 \\ 
0 & 0 & c_{15}^6-i(c_{24}^3+c_{14}^4) \\ 
0 & c_{15}^6+i(c_{24}^3+c_{14}^4) & -c_{14}^5-c_{13}^6+c_{24}^6-c_{23}^5
\end{pmatrix} .
\]
Therefore $d\hat{\rho}$ is semipositive if and only if $c_{15}^6=0$, $c_{24}^3=-c_{14}^4$
and $ -c_{14}^5-c_{13}^6+c_{24}^6-c_{23}^5>0$. In particular, $c_{15}^6=0$ and $c_{24}^3=-c_{14}^4$ imply that the Jacobi identities hold if and only if $c_{13}^4=c_{14}^4=0$.
However, this also implies $df^3=df^4=0$ so that $b_1(\mathfrak{g})\geq 4$ and we have to discard this case as well.
Therefore $c_{34}^5=c_{12}^6=0$ and, as a consequence,
\begin{equation}\label{b1=3}
\begin{aligned}
df^3 = & -c_{13}^3f^{13}-(c_{13}^4+c_{15}^6)f^{14}-c_{25}^4f^{15}-c_{23}^3 f^{23}-c_{24}^3 f^{24}-c_{25}^3 f^{25}, \\
df^4 = & -c_{13}^4f^{13}-c_{14}^4 f^{14}-c_{15}^4 f^{15}-c_{13}^3 f^{23}-(c_{13}^4 + c_{15}^6) f^{24}-c_{25}^4 f^{25},\\
df^5 = & -(c_{14}^6+c_{23}^6+c_{24}^5)f^{13}-c_{14}^5 f^{14}+(c_{14}^4+c_{13}^3) f^{15}-c_{23}^5 f^{23}-c_{24}^5 f^{24} \\ 
&+(c_{23}^3+  c_{13}^4 +c_{15}^6) f^{25}, \\
df^6= & -c_{13}^6 f^{13}-c_{14}^6 f^{14}-c_{15}^6 f^{15}-c_{23}^6 f^{23}-c_{24}^6 f^{24}-(c_{24}^3-c_{13}^3) f^{25}.
\end{aligned}
\end{equation}
In particular,  $f^{12}$   is a non-exact $2$-form  belonging to $\Lambda^2(\ker d)$ such that 
$f^{12} \wedge d\mathfrak{g}^*=0$.
On the other hand, a simple computation shows that for any Lie algebra $\mathfrak{g}_i$, for $i=11,14,21$, a $2$-form
$\alpha\in \Lambda^2(\ker d)$ such that $\alpha\wedge d\mathfrak{g}_i^*=0$ is necessarily exact, so we get a contradiction.
This concludes the non-existence part of the proof in the case $b_1=3$.

Now we consider the remaining case $b_1(\mathfrak{g})\geq 4$.
An explicit example of mean convex half-flat structure on $\mathfrak{g}_i$, $i=25,28,29,30,31,32,\allowbreak33$, is given in Table \ref{table2}.
Then, we only need to prove the non-existence of  mean convex half-flat structures on $\mathfrak{g}_{27}$. 

Let $(\omega,\rho)$ be a mean convex half-flat structure on $\mathfrak{g}_{27}$. 
We claim that on $\mathfrak{g}_{27}$ there exists an adapted basis  $(f^i)$ such that $df^1=df^2=df^3=0$ and $f_{6}\in \xi(\mathfrak{g}_{27})$. 
By Lemma \ref{U3=2,U2=2}, we can assume $U_2=U_3$ with $\dim_{\mathbb{R}}U_3=2$. 
We recall that $U_4$ has dimension $2$ or $4$.
Let us suppose $\dim_{\mathbb{R}}U_4=4$, first. We note that in this case the existence of an adapted basis $(f^i)$ for $(\omega,\rho)$   such that $f_6\in \xi(\mathfrak{g}_{27})$ and $V_4=U_4=\left<f^1,f^2,f^3,f^4\right>$ follows from the previous discussion on $\mathfrak{g}_{21}$, where we only used $\dim_{\mathbb{R}}U_2=2$ and $\dim_{\mathbb{R}}U_4=4$. In particular, since $V_4=\ker d$ on $\mathfrak{g}_{27}$, in this case we also have $df^1=df^2=df^3=df^4=0$.
When $\dim_{\mathbb{R}}U_4=2$ instead, since $U_2=U_3=U_4$, the discussion is the same as for $\mathfrak{g}_{11}$ and $\mathfrak{g}_{14}$, where we only used $U_3=U_4$ to find an adapted basis such that $df^1=df^2=0$  and $f_6$ lying in the center. In particular, since by construction $f^1,f^2,f^3\in V_4$, on $\mathfrak{g}_{27}$ we also have $df^3=0$, since $V_4=\ker d$. This proves our claim on $\mathfrak{g}_{27}$.
Now, using this claim  we shall show that  $\mathfrak{g}_{27}$ does not admit  any  mean convex half-flat structures.
Like in the previous cases, by contradiction, let us suppose there exists a nilpotent Lie algebra $\mathfrak{g}$ isomorphic to $\mathfrak{g}_{27}$  admitting   a mean convex half-flat structure $(\omega,\rho)$.  Then we may assume that   there exists  on $\frak g$ an adapted basis   $(f^i)$  for $(\omega,\rho)$  such that $df^1=df^2=df^3=0$ and $V_5=\left<f^1,f^2,f^3,f^4,f^5\right>$, so that $f_6\in \xi(\mathfrak{g})$.
Then 
\[ \displaystyle df^k=-\sum_{\substack{ i<j\\ i,j=1}}^5 c_{ij}^k f^{ij}, \quad k=4,5,6.
\]
By imposing the unimodularity of $\frak g$ and that $(\omega,\rho)$  is half-flat, we get
\begin{equation}\label{b1=4}
\begin{aligned}
df^4 = & c_{15}^6 f^{13}-c_{14}^4 f^{14}-c_{15}^4 f^{15},\\
df^5 = & c_{34}^5 f^{12}-(c_{24}^5+c_{14}^6+c_{23}^6) f^{13}-c_{14}^5 f^{14}+c_{14}^4 f^{15}-c_{23}^5 f^{23} \\ 
&-c_{24}^5 f^{24}-c_{34}^5 f^{34}, \\
df^6= & -c_{12}^6 f^{12}-c_{13}^6 f^{13}-c_{14}^6 f^{14}-c_{15}^6 f^{15}-c_{23}^6 f^{23}-c_{24}^6 f^{24}+c_{12}^6 f^{34}.
\end{aligned}
\end{equation}
Since  $b_1 (\frak g) =4$,  there  should exist a closed 1-form linearly independent from $f^1, f^2$ and $f^3$. Moreover,  since $\ker d=V_4\subset V_5= \left<f^1,f^2,f^3,f^4,f^5\right>$, the matrix $C$ associated to
\[
d:\left<f^4,f^5\right> \to \Lambda^2V_5 = \Lambda^2\left<f^1,f^2,f^3,f^4,f^5\right>
\]
 must have rank equal to $1$. This is equivalent to requiring that  $C$ is not the zero matrix  and all the $2\times 2$ minors of $C$ vanish. After eliminating all the zero rows,
 we have
\[
C=\begin{pmatrix}
0 & c_{34}^5 \\ 
c_{15}^6 & -c_{24}^5-c_{14}^6-c_{23}^6 \\ 
-c_{14}^4 & -c_{14}^5 \\ 
-c_{15}^4 & c_{14}^4 \\ 
0 & -c_{23}^5 \\ 
0 & -c_{24}^5 \\ 
0 & -c_{34}^5
\end{pmatrix} .
\]
By using that $(f^i)$ is an adapted basis and  \ref{exprbeta}, we get
\[
 (\beta_{m\overline{n}})  = \begin{pmatrix}
0 & 0 & 0 \\ 
0 & c_{15}^4 & c_{15}^6-ic_{14}^4 \\ 
0 & c_{15}^6+ic_{14}^4 & -c_{14}^5-c_{13}^6+c_{24}^6-c_{23}^5
\end{pmatrix}.
\]
Let us suppose $c_{15}^4=0$. Then $(\beta_{m\overline{n}})$ being positive semi-definite implies $c_{14}^4=c_{15}^6=0$, from which it follows that $\mathfrak{g}$ is $2$-step nilpotent, so that we can discard this case since $\mathfrak{g}_{27}$ is $3$-step nilpotent.   Thus,  we have to impose $c_{15}^4\neq 0$.
As a consequence,  $d^2f^i=0$, $i =4,5,6$, if and only if 
$c_{24}^5=c_{34}^5=c_{24}^6=c_{23}^5=c_{12}^6=0,$
from which it follows that $b_1(\mathfrak{g})=4$ holds if and only if
\[
c_{14}^5=-\frac{c_{14}^4}{c_{15}^4}, \quad c_{14}^6=\frac{c_{14}^4 c_{15}^6-c_{15}^4 c_{23}^6}{c_{15}^4}.
\]
Then  $\frak g$ must have structure equations  
\begin{equation}\label{g27}
\begin{aligned}
df^1=& df^2=df^3=0, \\
df^4=& c_{15}^6 f^{13}-c_{14}^4 f^{14}-c_{15}^4 f^{15}, \\
df^5=& -\frac{c_{14}^4 c_{15}^6}{c_{15}^4} f^{13}+\frac{(c_{14}^4)^2}{c_{15}^4} f^{14} + c_{14}^4 f^{15}, \\
df^6=& -c_{13}^6 f^{13}-\frac{c_{14}^4 c_{15}^6-c_{15}^4 c_{23}^6}{c_{15}^4} f^{14}-c_{15}^6 f^{15}-c_{23}^6 f^{23}.
\end{aligned}
\end{equation}
Note that, by \ref{g27},  $\mathfrak{g}$ has the same central and derived series as $\mathfrak{g}_{27}$  and, if  $c_{23}^6=0$, $\mathfrak{g}$ is almost abelian,   so it cannot be isomorphic to $\mathfrak{g}_{27}$.  Thus  we can suppose $c_{23}^6\neq 0$.
By \cite{Conti},  a $6$-dimensional 3-step nilpotent Lie algebra having $b_1=4$ and admitting a half-flat  structure must be isomorphic to either  $\mathfrak{g}_{25}$  or  $\mathfrak{g}_{27}$. 
In addition, $b_2(\mathfrak{g}_{25})=6$, while $b_2(\mathfrak{g}_{27})=7$. We shall show that  we cannot have  $b_2(\mathfrak{g})=7$ and so we shall get a contradiction. To this aim we need to compute the space  $Z^2$ of  closed $2$-forms. By a direct computation using  \ref{g27} and $c_{23}^6\neq 0$,   it follows that 
 $\dim Z^2=\dim \Lambda^2 V_4+2=8$.
Therefore, in order to get $b_2(\mathfrak{g})=7$, we have to require that  the space $B^2$  of exact 2-forms  is one-dimensional. 
This is equivalent to asking that the linear map
\[
d\rvert_{\left< f^4,f^5,f^6\right>}:\left< f^4,f^5,f^6\right> \to \Lambda^2 \mathfrak{g}^*,
\]
has rank equal to $1$.
Let us denote by $E$ the matrix associated to $d\rvert_{\left< f^4,f^5,f^6\right>}$ in the induced basis ($f^{ij}$) of $\Lambda^2 \mathfrak{g}^*$. Eliminating all the zero rows, one has
\[
E=\begin{pmatrix}
c_{15}^6 & -\dfrac{c_{14}^4 c_{15}^6}{c_{15}^4} & -c_{13}^6 \\ 
-c_{14}^4 & \dfrac{(c_{14}^4)^2}{c_{15}^4} &-\dfrac{c_{14}^4 c_{15}^6-c_{15}^4 c_{23}^6}{c_{15}^4} \\ 
-c_{15}^4 &  c_{14}^4 & -c_{15}^6 \\ 
0 & 0 & -c_{23}^6
\end{pmatrix} .
\]
Then $E$ has rank $1$ if and only if  $E$ is not the zero matrix and all the $2\times 2$ minors of $E$ vanish.
Notice that  the  minor  $c_{23}^6 c_{15}^4$ is different from zero,  since we have already excluded both cases $c_{23}^6=0$ and $c_{15}^4=0$. Then $\mathfrak{g}$ cannot be isomorphic to $\mathfrak{g}_{27}$ and we obtain  a contradiction. This concludes the case $b_1\geq 4$ and the proof of the theorem.
\end{proof}

\begin{remark}
By Theorem \hyperref[half-flat mean convex]{B}, we notice that, on a $6$-dimensional nilpotent Lie algebra $\mathfrak{g}$ with $b_1(\mathfrak{g})=2$, whenever a mean convex half-flat $\text{SU}(3)$-structure exists, a double example can also be found (see Table \ref{table2}). This is not true for different values of the first Betti number.
\end{remark}

Under the hypothesis of exactness, we can prove the following

\begin{theorem}\label{coupledmeanconvex} 
Let $\mathfrak{g}$ be a $6$-dimensional nilpotent Lie algebra admitting 
an exact mean convex $ \normalfont \text{SL}(3,\mathbb{C})$-structure. Then $\mathfrak{g}$ is isomorphic to $\mathfrak{g}_{18}
$ or $\mathfrak{g}_{28}$. Moreover, up to a change of sign, every  exact definite $3$-form $\rho$ on $\mathfrak{g}_{18}$ and $\mathfrak{g}_{28}$  is mean convex, and $\mathfrak{g}_{28}$ is the only nilpotent Lie algebra admitting  mean convex  coupled  structures, up to isomorphism.
\end{theorem}

\begin{proof}
Among the $6$-dimensional nilpotent Lie algebras admitting half-flat structures, as shown in the proof of \cite[Theorem 4.1]{FinoRaffero}, the only Lie algebras that can admit exact $\text{SL}(3,\mathbb{C})$-structures are isomorphic to 
$\mathfrak{g}_4$,
$\mathfrak{g}_9$ or $\mathfrak{g}_{28}$. 
Therefore, by Theorem \hyperref[theoremA]{A}, $\mathfrak{g}_{28}$ is the only nilpotent Lie algebra among them which can admit a mean convex structure. 
In particular, a coupled  mean convex structure on $\mathfrak{g}_{28}$ is given in Table \ref{table2}.
This example was first found in \cite{FinoRaffero}, up to a change of sign of the definite $3$-form.
For the remaining nilpotent Lie algebras $\mathfrak{g}_i$, for $i=3,5,17,18,19,20,23,26,$ which can admit mean convex $\text{SL}(3,\mathbb{C})$-structures by Theorem \hyperref[theoremA]{A}, we prove that $\mathfrak{g}_{18}$ is the only one that admits exact definite $3$-forms. To see this, let $(e^j)$ be the basis of $\mathfrak{g}_{i}^*$ as listed in Table \ref{table1}. Then  the generic  exact $3$-form $\rho$ on $\frak g_i$  is given by $d \eta$, where 
\begin{equation}\label{2-form}
\eta=\sum_{i<j}p_{ij}e^{ij}, \quad p_{ij}\in \mathbb{R}.
\end{equation}
By an explicit computation, one can show that, on $\mathfrak{g}_i$, for $i=3,17,19,23,26$,  $\lambda(\rho)=0$, while, on $\mathfrak{g}_{5}$ and $\mathfrak{g}_{20}$, $\lambda(\rho)=p_{56}^4>0$. Finally, on $\mathfrak{g}_{18}$, $\lambda(\rho)=-4 p_{56}^4$. Then, if $p_{56}\neq 0$, $\rho=d\eta$ is a definite $3$-form on $\mathfrak{g}_{18}$.
Moreover, $(e^1-iJ_{\rho}e^1, e^3-iJ_{\rho}e^3, e^5-iJ_{\rho}e^5)$ is a basis for $\Lambda^{1,1}\mathfrak{g}_{18}^*$ and, with respect to this basis, the matrix $(\beta_{m\overline{n}})$ associated to the $(2,2)$-form $d \hat{\rho}$ is
$\text{diag}(0,0,-4p_{56})$. Then, when $p_{56}<0$, $\rho$ is mean convex, otherwise $-\rho$ is.  
By a direct computation one can check that the same conclusions hold also for $\mathfrak{g}_{28}$. In particular, the generic  exact $3$-form $\rho=d\eta$, with $\eta$ as in \ref{2-form}, is definite as long as $p_{56}\neq0$. Moreover, $(e^1-iJ_{\rho}e^1, e^3-iJ_{\rho}e^3, e^5-iJ_{\rho}e^5)$ is a basis  of $\Lambda^{1,1}\mathfrak{g}_{28}^*$,  for every exact definite $\rho$  and, with respect to this basis, the matrix $(\beta_{m\overline{n}})$ associated to the $(2,2)$-form $d \hat{\rho}$ is
$\text{diag}(0,0,-4p_{56})$.
\end{proof}

\section{Hitchin flow equations} \label{Section6}

In this section we study the mean convex property in relation to the Hitchin flow equations \ref{HitchinFlow}.
We recall that  the solution $(\omega(t), \rho(t))$ of \ref{HitchinFlow}  starting from a half-flat structure remains half-flat as long as it exists.  However, the same does not happen in general for special classes of half-flat structures. Then, a natural question  is  whether the Hitchin flow equations preserve the mean convexity of the initial data $(\omega(0), \rho(0))$.
A first example of solution  preserving the mean convex condition of the initial data, up to change of sign of $\rho(0)$,  was found in \cite[Proposition 5.4]{FinoRaffero2}. In this case the   initial structure is coupled.

More generally, when the Hitchin flow solution $(\omega(t),\rho(t))$ preserves the coupled condition of the initial data,
then $\rho(t)=f(t)\rho(0)$, where $f\colon I\to \mathbb{R}$ is a non-zero smooth function with $f(0)=1$ (for more details see \cite[Proposition 5.2]{FinoRaffero2}). Then, a coupled solution preserves the mean convexity of the initial data as long as it exists.

Some further remarks can be made in other special cases.  
If $(\omega(t),\rho(t))$ is a solution of \ref{HitchinFlow}  starting from a strictly mean convex half-flat structure $(\omega,\rho)$, by continuity the solution remains mean convex, at least for small times.
This occurs, for instance, for double structures.
In particular cases, the mean convex property of the double initial data is preserved for all times:

\begin{proposition}
Let $M$ be a connected $6$-manifold endowed with a double structure $(\omega,\rho)$.
If $(\omega(t),\rho(t))$ is a double solution of \ref{HitchinFlow} defined on some $I\subseteq \mathbb{R}$, $0\in I$,  i.e. $d\hat{\rho}(t)=\nu_0(t)\omega^2(t)$ for each $t\in I$ for some  smooth nowhere vanishing function $\nu_0\colon I\to \mathbb{R} $, then there exists a nowhere vanishing smooth function $f:I\to \mathbb{R}$ such that $\omega(t)=f(t)\omega(0)$. 
Conversely, if  $(\omega(t),\rho(t))$ is a solution of \ref{HitchinFlow} with $\omega(t)=f(t)\omega(0)$, then it is a double solution. 
\end{proposition}

\begin{proof}
Let $(\omega(t),\rho(t))$  be a solution with $\omega(t)=f(t)\omega(0)$.
From \ref{HitchinFlow} one gets
\[
d\hat{\rho}(t)=-\frac{1}{2}\frac{\partial}{\partial t}\left(\omega(t)^2\right) =-\frac{1}{2}\frac{\partial}{\partial t}\left( f^2(t)\omega(0)\wedge \omega(0)\right)=-f(t)\dot{f}(t)\omega(0)^2.
\]
Then $\omega(t)=f(t)\omega(0)$ is a double solution with $\nu_0(t)=- \frac{d}{d t} \ln f(t)$.
Conversely, if $d\hat{\rho}(t)=\nu_0(t)\omega(t)^2$,
then 
\[
\frac{\partial}{\partial t} \omega(t)\wedge \omega(t)=-d\hat{\rho}(t)=-\nu_0(t)\omega(t)^2.
\]
Since the wedge product with $\omega(t)$ is injective on $2$-forms, this is equivalent to $\frac{\partial}{\partial t}\omega(t)=-\nu_0(t)\omega(t)$, whose unique solution is $\omega(t)=f(t)\omega(0)$, with $f(t)=e^{-\int_0^t \nu_0(s)ds}$.
\end{proof}

We now provide an explicit example of double solution to \ref{HitchinFlow} and show that a double solution with double initial data may not exist.

\begin{example}
Consider the double $\text{SU}(3)$-structure $(\omega,\rho)$ given in Table \ref{table2} on $\mathfrak{g}_{24}$.
The solution of the Hitchin flow equations with initial data  $(\omega,\rho)$ is double and it is explicitly given by
\begin{align*}
\omega(t)&=\left( 1-\frac{5}{2}t \right)^{\frac{1}{5}}\omega, \\
\rho(t)&=-\left( 1-\frac{5}{2}t \right)^{\frac{6}{5}}e^{123}+e^{145}+e^{246}+e^{356}.
\end{align*}
In particular  $d\hat{\rho}(t)=\nu_0(t)\omega^2(t)$ with   $\nu_0(t) = (2-5t)^{-1}>0$  for each $t$ in the maximal interval of definition $I=(-\infty,\frac{2}{5})$.
Consider now the double $\text{SU}(3)$-structure $(\omega,\rho)$ given in Table \ref{table2} on $\mathfrak{g}_{6}$.
The solution of the Hitchin flow equation with initial data  $(\omega,\rho)$ is given by
\begin{align*}
\omega(t)&=f_1(t)\left( e^{15}-e^{24} \right)-f_2(t)e^{36}, \\
\rho(t)&=h_1(t)e^{123}+\left(h_2(t)-1\right)e^{134}-e^{146}-e^{235}+e^{256}-e^{345}+h_2(t)e^{126},
\end{align*}
where $f_1(t),f_2(t),h_1(t),h_2(t)$ satisfy the following  autonomous \textsc{ode} system:
\[
\begin{cases}
\dot{f_1} =\frac{1}{2f_1^3f_2}\left( 2h_2-1 \right), \\[2pt]
\dot{f_2} =-\frac{1}{2f_1^4f_2}\left( 2f_1+f_2\left(2h_2-1\right)\right), \\[2pt]
\dot{h_1}=-2f_1,\\[2pt]
\dot{h_2}=-f_2,
\end{cases}
\]
with initial conditions $f_1(0)=f_2(0)=h_1(0)=1$, $h_2(0)=0$,
which, by known theorems, admits a unique solution with given initial data. 
In particular,  this solution is not a double solution.
A direct computation shows that the eigenvalues $\lambda_i(t) $ of the matrix $(\beta_{m\overline{n}}(t))$ associated to $d\hat{\rho}(t)$ are
\[
\lambda_1=\lambda_2=\sqrt{-h_2^2+h_1+h_2}, \quad
\lambda_3=(1-2h_2)\sqrt{-h_2^2+h_1+h_2}.
\]
In particular the mean convex property is preserved for small times as expected.
\end{example}

To our knowledge, the question of whether the Hitchin flow preserves the mean convexity of the initial data when the $(2,2)$-form is not positive but just semi-positive is still open.
Nonetheless, some easy considerations can be made in order to obtain a better understanding of the problem.
Let $M$ be a compact real analytic $6$-dimensional manifold endowed with a half-flat mean convex $\text{SU}(3)$-structure $(\omega,\rho)$.
Since the unique solution of
\ref{HitchinFlow} starting from $(\omega,\rho)$ is a  one-parameter family of half-flat structures $(\omega(t),\rho(t))$, we can write
\[
d\hat{\rho}(t)=(\nu_0(t)\omega(t)-\nu_2(t))\wedge\omega(t),
\]
where $\nu_0(t)\in C^{\infty}(M)$ and $\nu_2(t)\in \Lambda^{1,1}_0M$ is a primitive $(1,1)$-form with respect to $J_{\rho(t)}$ for each $t\in I$, where $I$ is the maximal interval of definition of the flow. 
Then $d\hat{\rho}(t)\wedge\omega(t)=\nu_0(t)\omega(t)^3$
and, since $\nu_0(0)>0$ by the mean convexity of the initial data, by continuity we have $\nu_0(t)>0$ at least for small times.
By \ref{HitchinFlow}, as long as $\nu_0(t)>0$, the volume form $\omega(t)^3$ is pointwise decreasing:
\[\frac{\partial}{\partial t} (\omega(t)^3)=\frac{\partial}{\partial t}(\omega(t)^2)\wedge\omega(t)+
\frac{\partial}{\partial t}\omega(t)\wedge \omega(t)^2=-3d\hat{\rho}(t)\wedge\omega(t)=-3\nu_0(t)\omega(t)^3.
\]
Moreover, $\omega(t)^2$ is a positive $(2,2)$-form with respect to $J_{\rho(t)}$ for all $t\in I$ and, from the second equation in \ref{HitchinFlow}, we know that $-\partial_t(\omega^2(t))$ remains a $(2,2)$-form with respect to $J_{\rho(t)}$ for each $t\in I$ such that $-\partial_t(\omega^2(t))\big|_{t=0}=2 d\hat{\rho}(0)$ is semi-positive.
Then the Hitchin flow solution preserves the mean convexity of the initial data if and only if $-\partial_t(\omega^2(t))=2 d\hat{\rho}(t)$ remains semi-positive. 
The essential difficulty in this problem lies in the fact that the link between the positivity of $\omega^2(t)$ and the mean convexity of the initial data is not sufficient to ensure the mean convexity of the solution since also the almost complex structure evolves in a non-linear way under the equation $\partial_t(\rho(t))=d\omega(t)$.
Let us look at the behaviour of \ref{HitchinFlow} on a specific example.

\begin{example}
Consider the  mean convex half-flat structure $(\omega,\rho)$ given in Table \ref{table2} on $\mathfrak{g}_{25}$ and consider the family of solutions to the second equation in \ref{HitchinFlow}, starting from $(\omega,\rho)$:
\begin{align*}
\omega(t)&=-a_1(t)e^{13}+\frac{1}{a_2(t)}e^{45}+a_2(t)e^{26}, \\
\rho(t)&=e^{156}+b_1(t)e^{124}-e^{235}-e^{346}+b_2(t)(e^{125}-e^{234}),
\end{align*}
where $a_1(t),a_2(t),b_1(t),b_2(t)$ satisfy the following \textsc{ode} system:
\begin{equation}\label{SecondEquation}
\begin{cases}
\dot{a_1}=-\frac{1}{2a_1 a_2}\left(2a_2^2b_2+1\right), \\
\dot{a_2}=\frac{1}{2a_1^2}\left(2a_2^2b_2-1\right), \\
\end{cases}
\end{equation}
subject to the normalization condition $\sqrt{b_1-b_2^2}=a_1$, with initial data $a_1(0)=a_2(0)=b_1(0)=1$,
 $b_2(0)=0$.
This system defines a family of solutions to $\frac{1}{2}\partial_t(\omega(t)^2)=-d\hat{\rho}(t)$ depending on $b_2(t)$. Then, if $b_2(t)=a_1(t)-1$, for instance, $d\hat{\rho}(t)$ is not semi-positive, at least for small times $t>0$.  
Anyway, the unique solution to \ref{HitchinFlow}  starting from $(\omega,\rho)$, given by \ref{SecondEquation} together with 
\[
\begin{cases}
\dot{b_1}=-\frac{1}{a_2}, \\
\dot{b_2}=a_2,
\end{cases}
\]
preserves the mean convexity of the initial data.
\end{example}
By a direct computation, one can show that the mean convexity of the initial data is preserved by \ref{HitchinFlow}, for small times, also in all the other examples of half-flat mean convex structures given in Table \ref{table2}.

\section{Proof of Theorem C}\label{Section7}

We recall that a symplectic form $\Omega$ is said to \emph{tame} an almost complex structure $J$ if its $(1,1)$-part $\Omega^{1,1}$ is positive.
A closed $\text{SL}(3,\mathbb{C})$-structure $\rho$ is then called   \emph{tamed} if there exists a symplectic form $\Omega$ taming  the induced almost complex structure $J_{\rho}$.
  As  already observed in \cite{Donaldson},   compact $6$-manifolds cannot  admit   tamed mean convex $\text{SL}(3,\mathbb{C})$-structures. 

Notice that, if we denote as usual  $\hat{\rho}= J_\rho \rho$,  when the normalization condition $\rho\wedge\hat{\rho}=\frac{2}{3}\omega^3$ is satisfied and $d \omega =0$, then  the pair $(\omega,\rho)$ defines  a symplectic half-flat structure.

Since we consider invariant   tamed closed $\text{SL}(3,\mathbb{C})$-structures on solvmanifolds, we can work as in the previous sections at the level of solvable unimodular   Lie algebras.

\begin{proof}[Proof of Theorem C]  First we prove the theorem in the nilpotent case. $6$-dimensional symplectic nilpotent  Lie algebras   were classified in \cite{Goze} (see also \cite{salamon}) and their structure equations 
 are listed in Table \ref{table1}.
For any such Lie algebra we consider a pair $(\rho,\Omega)\in \Lambda^3\mathfrak{g}_i^*\times \Lambda^2\mathfrak{g}_i^*$ explicitly given by
\[
\rho=\sum_{i<j<k}p_{ijk}\, e^{ijk}, \quad \Omega=\sum_{r<s}h_{rs}\, e^{rs}, 
\]
where $p_{ijk}, h_{rs}\in \mathbb{R}$, and impose the two conditions $d\rho=0$  and  $d\Omega=0$, which are both linear in the coefficients $p_{ijk}, h_{rs}$. Then $\Omega$ is a symplectic form provided that  it is non-degenerate, i.e. $\Omega^3\neq 0$. 
By  \cite[Lemma 3.1]{EnriettiFinoVezzoni}, 
a real Lie algebra $\mathfrak{g}$ endowed with an almost complex structure $J$ such that $J\xi(\mathfrak{g}) \cap [\mathfrak{g},\mathfrak{g}] \neq \{0\} $, $\xi(\mathfrak{g})$ being the center of $\mathfrak{g}$, cannot admit a symplectic form $\Omega$  taming $J$. If we assume $\lambda(\rho)<0$, we may then apply this result on each $\mathfrak{g}_i$ by considering the almost complex structure $J_{\rho}$ induced by $\rho$.
We notice that, for any $\mathfrak{g}_i$ listed in Table \ref{table1}, $e_6\in \xi(\mathfrak{g}_i)$. 
A direct computation on each $\mathfrak{g}_i$ for $i=3,4,5,6,7,8,9,10,13,18,19,20,28,29,30$, shows that $J_{\rho}e_6\in [\mathfrak{g}_i,\mathfrak{g}_i]$,  for any $J_{\rho}$ induced by a closed  $3$-form $\rho$.
On $\mathfrak{g}_i$, for $i=23,26,33$, the same obstruction holds since an explicit computation shows that the map 
\[
\pi\circ J_{\rho}: \xi(\mathfrak{g}_i) \to \mathfrak{g}_i, 
\]
has non-trivial kernel, where $\pi$ denotes the projection onto  $\mathfrak{g}_i/[\mathfrak{g}_i,\mathfrak{g}_i]$. This means that, for each $\rho$, one can find a non-zero element in the center of $\mathfrak{g}_i$ whose image under $J_{\rho}$ lies entirely in $[\mathfrak{g}_i,\mathfrak{g}_i]$.
For all the other cases, let $\Omega=\Omega^{1,1}+\Omega^{2,0}+\Omega^{0,2}$ be the decomposition of $\Omega$ in types with respect to $J_{\rho}$, and denote by $\omega$ the $(1,1)$-form $\Omega^{1,1} \coloneqq \frac{1}{2}\left(\Omega+J_{\rho}\Omega\right)$. Then, in order to have a closed $\text{SL}(3,\mathbb{C})$-structure tamed by $\Omega$ we have to require that $\omega$ is positive, i.e., that the symmetric $2$-tensor $g\coloneqq\omega(\cdot,J_{\rho}\cdot)$ is positive definite. Denote by $g_{ij}\coloneqq g(e_i,e_j)$ the coefficients of $g$ with respect the dual basis $\left(e_1,\ldots,e_6\right)$ of $\mathfrak{g}$.
Then, a direct computation on $\mathfrak{g}_i$, for $i=11,12,21,22,27$, shows that $g_{66}$ always vanishes, so we may discard  these cases as well. 
We may then restrict our attention to the remaining Lie algebras $\mathfrak{g}_{24}$ and $\mathfrak{g}_{31}$.
Since, as shown in \cite[Theorem 2.4]{ContiTomassini}, these are the only $6$-dimensional non-abelian nilpotent Lie algebras  carrying a symplectic half-flat structure.
Explicit examples of closed  $\text{SL}(3,\mathbb{C})$-structures tamed  by a symplectic form $\Omega$ such that $d\Omega^{1,1}\neq 0$ are given by
\[\rho= -e^{125}-e^{146}-e^{156}-e^{236}-e^{245}-e^{345}-e^{356}, \quad \Omega= e^{13}+\frac{1}{2} e^{14}-\frac{1}{2} e^{24}+e^{26}+e^{35}+e^{36}, \]  on $\mathfrak{g}_{24}$, 
and by
\[
\rho= e^{123}+2e^{145}+e^{156}+e^{235}+e^{246}+e^{345}, \quad
\Omega=e^{16}-e^{25}-e^{34}+e^{36},
\] on $\mathfrak{g}_{31}$. This proves the first part of the theorem.

Using the classification results  in  \cite[Th.  2]{Macri} for  $6$-dimensional symplectic unimodular   (non-nilpotent) solvable Lie algebras, for each Lie algebra one can compute the metric coefficients $g_{ij}$  of $g$ with respect to the basis $(e_1,\ldots,e_6)$ for $\mathfrak{g}$ as listed in Table \ref{table3}.
It turns out that, if  $\mathfrak{g}$ is one among
$\mathfrak{g}_{6,3}^{0,-1}$, $\mathfrak{g}_{6,10}^{0,0}$, $\mathfrak{g}_{6,13}^{-1,\frac{1}{2},0}$, $\mathfrak{g}_{6,13}^{\frac{1}{2},-1,0}$, $\mathfrak{g}_{6,21}^0$, $\mathfrak{g}_{6,36}^{0,0}$, $\mathfrak{g}_{6,78}$, $A_{5,8}^{-1} \oplus \mathbb{R}$, $A_{5,13}^{-1,0,\gamma}$, $A_{5,14}^{0} \oplus \mathbb{R}$, $A_{5,15}^{-1} \oplus \mathbb{R}$, $A_{5,17}^{0,0,\gamma}\oplus\mathbb{R},$ $A_{5,18}^{0} \oplus \mathbb{R}$ or $A_{5,19}^{-1,2} \oplus \mathbb{R}$, each closed definite $3$-form $\rho$ induces a $J_{\rho}$ such that $g_{11}=0$. In a similar way,
if $\mathfrak{g}$ is $\mathfrak{g}_{6,15}^{-1}$ or $\mathfrak{g}_{6,18}^{-1,-1}$, then $g_{44}=0$, while when $\mathfrak{g}$ is $\mathfrak{n}_{6,84}^{\pm 1}$,  $\mathfrak{e}(2) \oplus \mathbb{R}^3$ or $\mathfrak{e}(1,1) \oplus \mathbb{R}^3$,  $g_{33}=0$. Finally, when $\mathfrak{g}=\mathfrak{e}(1,1)\oplus \mathfrak{h}$, then $g_{66}=0$.  
In some other cases $g$ cannot ever be positive definite since, for each closed $\rho$ inducing an almost complex structure $J_{\rho}$,  $g_{rr}=-g_{kk}$ for some $r\neq k$.
In particular, when $\mathfrak{g}=\mathfrak{g}_{6,70}^{0,0}$, then $g_{11}=-g_{22}$, when $\mathfrak{g}=\mathfrak{e}(2)\oplus \mathfrak{e}(2)$, then $g_{55}=-g_{66}$, and when $\mathfrak{g}$ is $\mathfrak{e}(2)\oplus\mathfrak{e}(1,1)$ or $\mathfrak{e}(2)\oplus\mathfrak{h}$, then $g_{22}=-g_{33}$. 
As shown in \cite[Prop.  3.1, 4.1 and 4.3]{SHFsolvmanifolds}, for the remaining Lie algebras
$\mathfrak{g}_{6,38}^0$, $\mathfrak{g}_{6,54}^{0,-1}$, $\mathfrak{g}_{6,118}^{0,-1,-1}$, $\mathfrak{e}(1,1) \oplus \mathfrak{e}(1,1)$, $A_{5,7}^{-1,\beta,-\beta}$, $A_{5,17}^{0,0,-1} \oplus \mathbb{R}$,  $A_{5,17}^{\alpha,-\alpha,1} \oplus \mathbb{R}$, as listed in Table \ref{table3},
a symplectic half-flat structure always exists. Moreover, on these Lie algebras, an explicit example of closed $ \text{SL}(3,\mathbb{C})$-structure tamed by a symplectic form $\Omega$ such that $d\Omega^{1,1}\neq 0 $ is given Table \ref{table3}.
\end{proof}

\begin{remark}
\begin{enumerate}

\item   By  \cite[Remarks 3.2 and 4.4]{SHFsolvmanifolds}, the solvable Lie groups corresponding to each solvable Lie algebra admitting closed tamed $\text{SL}(3,\mathbb{C})$-structures admit compact quotients by lattices (for further details see \cite{Bock, FernandezLeonSaralegui, TralleOprea, Yamada}).

\item  As shown in \cite{Donaldson}, given  an $\text{SL}(3,\mathbb{C})$-structure $\rho$ tamed by a $2$-form $\Omega$ on a real $6$-dimensional vector space $V$, the $3$-form  \[
\varphi=\rho+\Omega\wedge dt,
\]
defines a $\text{G}_2$- structure on  $V\oplus \mathbb{R}$.
Therefore, as an application of Theorem \hyperref[tamed]{C}, we classify decomposable solvable Lie algebras of the form  $\mathfrak{g} \oplus \mathbb{R}$ admitting a closed $\text{G}_2$-structure. In particular, in the nilpotent case, this result was already obtained in \cite{ContiFernandez}.  
\end{enumerate}
 \end{remark}

\newpage

\section*{Appendix}

Table  \ref{table1}  contains the  isomorphism classes  of  $6$-dimensional real nilpotent Lie algebras  $\mathfrak{g}_i$, $i =1, \ldots, 34,$  including  their  first Betti numbers and an indication of whether they admit half-flat structures and symplectic forms.
In Table \ref{table2} we give  an explicit example of mean convex closed $\text{SU}(3)$-structure, indicating   which  ones are half-flat.  
Table \ref{table3} contains all $6$-dimensional  symplectic solvable (non-nilpotent) unimodular Lie algebras, specifying  which admit    tamed  closed $\text{SL}(3,\mathbb{C})$-structures. An explicit example of  a  closed tamed $\text{SL}(3,\mathbb{C})$-structure is  also included.

\begin{table}[H]
\caption{$6$-dimensional real nilpotent Lie algebras} 
\label{table1}
\smallskip
\centering
\renewcommand\arraystretch{1.25}
\scalebox{0.86}{
\begin{tabular}{|c|c|c|c|c|}
\hline
$\mathfrak{g}$ & {\normalfont Structure constants } &$b_1(\mathfrak{g})$ & Half-flat & Symplectic\\ \hline
$\mathfrak{g}_1$       &$(0,0,e^{12},e^{13}, e^{14}+e^{23}, e^{34}-e^{25})$& 2 & --& -- \\ \hline 
$\mathfrak{g}_2$       &$(0,0,e^{12},e^{13},e^{14},e^{34}-e^{25})$& 2& --& -- \\ \hline
$\mathfrak{g}_3$       &$(0,0,e^{12},e^{13},e^{14},e^{15})$& 2& -- & \cmark \\ \hline
$\mathfrak{g}_4$       &$(0,0,e^{12},e^{13},e^{14}+e^{23},e^{24}+e^{15})$& 2& \cmark & \cmark \\ \hline
$\mathfrak{g}_5$       &$(0,0,e^{12},e^{13},e^{14},e^{23}+e^{15})$& 2& -- & \cmark \\ \hline
$\mathfrak{g}_6$  &$(0,0,e^{12},e^{13},e^{23},e^{14})$& 2& \cmark  & \cmark \\ \hline
$\mathfrak{g}_7$  &$(0,0,e^{12},e^{13},e^{23},e^{14}-e^{25})$& 2& \cmark & \cmark \\ \hline
$\mathfrak{g}_8$  &$(0,0,e^{12},e^{13},e^{23},e^{14}+e^{25})$& 2& \cmark & \cmark \\ \hline
$\mathfrak{g}_9$  &$(0,0,0,e^{12},e^{14}-e^{23},e^{15}+e^{34})$& 3&\cmark  & \cmark \\ \hline
$\mathfrak{g}_{10}$  &$(0,0,0,e^{12},e^{14},e^{15}+e^{23})$& 3& \cmark  & \cmark \\ \hline
$\mathfrak{g}_{11}$  &$(0,0,0,e^{12},e^{14},e^{15}+e^{23}+e^{24})$& 3& \cmark &\cmark \\ \hline
$\mathfrak{g}_{12}$  &$(0,0,0,e^{12},e^{14},e^{15}+e^{24})$& 3& \cmark & \cmark \\ \hline
$\mathfrak{g}_{13}$  &$(0,0,0,e^{12},e^{14},e^{15})$& 3& \cmark & \cmark \\ \hline
$\mathfrak{g}_{14}$  &$(0,0,0,e^{12},e^{13},e^{14}+e^{35})$& 3& \cmark & -- \\ \hline
$\mathfrak{g}_{15}$  &$(0,0,0,e^{12},e^{23},e^{14}+e^{35})$& 3& \cmark & -- \\ \hline
$\mathfrak{g}_{16}$  &$(0,0,0,e^{12},e^{23},e^{14}-e^{35})$& 3& \cmark & -- \\ \hline
$\mathfrak{g}_{17}$  &$(0,0,0,e^{12},e^{14},e^{24})$& 3& -- & -- \\ \hline
$\mathfrak{g}_{18}$ & $(0,0,0,e^{12},e^{13}-e^{24},e^{14}+e^{23})$& 3  &-- & \cmark \\ \hline
$\mathfrak{g}_{19}$ & $(0,0,0,e^{12},e^{14},e^{13}-e^{24})$	& 3			& --&	\cmark	                 \\ \hline 
$\mathfrak{g}_{20}$ & $(0,0,0,e^{12},e^{13}+e^{14},e^{24})$ & 3 & -- & \cmark \\ \hline 
$\mathfrak{g}_{21}$ & $(0,0,0,e^{12},e^{13},e^{14}+e^{23})$& 3  & \cmark & \cmark \\ \hline 
$\mathfrak{g}_{22}$ & $(0,0,0,e^{12},e^{13},e^{24})$& 3 &\cmark & \cmark \\ \hline 
$\mathfrak{g}_{23}$ & $(0,0,0,e^{12},e^{13},e^{14})$ & 3&-- & \cmark\\ \hline
$\mathfrak{g}_{24}$ & $(0,0,0,e^{12},e^{13},e^{23})$ & 3& \cmark & \cmark \\ \hline 
$\mathfrak{g}_{25}$ & $(0,0,0,0,e^{12},e^{15}+e^{34})$ & 4 & \cmark & -- \\ \hline 
$\mathfrak{g}_{26}$ & $(0,0,0,0,e^{12},e^{15})$& 4 & -- & \cmark \\ \hline 
$\mathfrak{g}_{27}$  & $(0,0,0,0,e^{12},e^{14}+e^{25})$& 4& \cmark &\cmark  \\ \hline
$\mathfrak{g}_{28}$ & $(0,0,0,0,e^{13}-e^{24},e^{14}+e^{23})$& 4 & \cmark & \cmark \\ \hline 
$\mathfrak{g}_{29}$ & $(0,0,0,0,e^{12},e^{14}+e^{23})$ & 4& \cmark & \cmark \\ \hline
$\mathfrak{g}_{30}$ & $(0,0,0,0,e^{12},e^{34})$& 4 & \cmark & \cmark \\ \hline 
$\mathfrak{g}_{31}$ & $(0,0,0,0,e^{12},e^{13})$& 4 & \cmark &  \cmark \\ \hline 
$\mathfrak{g}_{32}$ & $(0,0,0,0,0,e^{12}+e^{34})$ & 5 &\cmark & -- \\ \hline 
$\mathfrak{g}_{33}$ & $(0,0,0,0,0,e^{12})$& 5 & \cmark & \cmark \\ \hline 
$\mathfrak{g}_{34}$  & $(0,0,0,0,0,0)$ & 6 & \cmark & \cmark \\ \hline
\end{tabular} }
\end{table}
\renewcommand\arraystretch{1}

\begin{table}[H] 
\caption{Explicit examples of mean convex closed $\text{SU}(3)$-structures}
\smallskip

\label{table2}
\centering
 \scalebox{0.70}{
\renewcommand\arraystretch{1.25}
\begin{tabular}{|c|c|c|}
\hline
$\mathfrak{g}$ & \normalfont{Mean convex closed $\text{SU}(3)$-structures }&  \normalfont{Half-flat mean convex example } \\ \hline 
$\mathfrak{g}_3$ & \begin{tabular}{@{}c@{}} 
 $ \omega=-e^{12}-e^{35}-e^{46}$ \\ $\rho=-\frac{5}{4}e^{136}+\frac{5}{4}e^{145}-e^{156}-e^{234}-e^{236}+e^{245}  $ \\
\end{tabular} & -- \\ \hline 
$\mathfrak{g}_5$ & \begin{tabular}{@{}c@{}} 
 $ \omega=-e^{12}-e^{35}-e^{46}$ \\ $\rho=\frac{1}{2}e^{134}-e^{156}-e^{236}+2e^{245} $\\
\end{tabular} & -- \\  \hline 
$\mathfrak{g}_6$  & \begin{tabular}{@{}c@{}} 
 $ \omega=e^{15}-e^{24}-e^{36}$ \\ $\rho=e^{123}-e^{134}-e^{146}-e^{235}-e^{256}-e^{345} $\\
\end{tabular} & \cmark\\ \hline
$\mathfrak{g}_7$   & \begin{tabular}{@{}c@{}} 
 $ \omega=-\frac{1}{2}e^{15}+\frac{1}{2}e^{24}-\frac{3}{2}e^{36}$ \\ $\rho=-\frac{3}{4}e^{123}+\frac{1}{3}e^{134}-e^{146}+\frac{1}{12}e^{235}-\frac{1}{4}e^{256}+\frac{3}{4}e^{345} $\\
\end{tabular}& \cmark \\ \hline
$\mathfrak{g}_8$  & \begin{tabular}{@{}c@{}} 
 $ \omega=e^{15}-e^{24}-\frac{1}{2}e^{36}$ \\ $\rho=e^{123}-e^{134}-\frac{1}{2}e^{146}-e^{235}-\frac{1}{2}e^{256}-e^{345} $\\
\end{tabular}& \cmark \\ \hline
$\mathfrak{g}_{10}$  & \begin{tabular}{@{}c@{}} 
 $ \omega=-\frac{1}{2}e^{13}+e^{46}-e^{25}$ \\ $\rho=e^{124}-e^{145}+e^{156}-\frac{1}{2}e^{234}-\frac{1}{2}e^{236}+\frac{1}{2}e^{345}   $\\
\end{tabular} & \cmark \\ \hline
$\mathfrak{g}_{11}$  &  \begin{tabular}{@{}c@{}} 
 $ \omega=\frac{5}{4}e^{13}+\frac{28}{3}e^{24}+e^{25}-\frac{82}{15}e^{26}+\frac{5}{4}e^{34}+e^{35}+e^{45}+\frac{14}{3}e^{46}+e^{56} $ \\ $\rho=2e^{125}+e^{126}-\frac{5}{4}e^{134}+e^{136}+e^{146}+e^{156}-e^{236}+e^{245}-e^{246}  $\\
 \end{tabular} & -- \\ \hline
$\mathfrak{g}_{13}$ & \begin{tabular}{@{}c@{}} 
 $ \omega=e^{13}+e^{46}+e^{25}$ \\ $\rho=-e^{124}+e^{145}+e^{156}+e^{234}-e^{236}-e^{345}   $\\
\end{tabular} & \cmark \\ \hline
$\mathfrak{g}_{14}$ & \begin{tabular}{@{}c@{}} 
 $ \omega=e^{13}-e^{26}+e^{45}$ \\ $\rho=-e^{125}-e^{146}+e^{234}+e^{356}  $\\
\end{tabular} & -- \\ \hline
$\mathfrak{g}_{15}$&  \begin{tabular}{@{}c@{}} 
 $ \omega=e^{15}+e^{34}-e^{26}$ \\ $\rho=e^{123}+e^{136}-e^{146}+e^{235}-e^{245}+e^{356}   $\\
\end{tabular} & \cmark \\ \hline
$\mathfrak{g}_{16}$   & \begin{tabular}{@{}c@{}} 
 $ \omega=e^{13}+e^{26}-e^{45}$ \\ $\rho=2e^{124}-\frac{\sqrt{2}}{2}e^{156}-e^{235}+\frac{\sqrt{2}}{2}e^{346}  $\\
\end{tabular} & \cmark \\ \hline
$\mathfrak{g}_{17}$  & \begin{tabular}{@{}c@{}} 
 $ \omega=e^{12}+e^{34}+e^{56}$ \\ $\rho=-e^{135}+2e^{146}+e^{236}+\frac{1}{2}e^{245}  $ \\
\end{tabular} & -- \\ \hline
 $\mathfrak{g}_{18}$ & \begin{tabular}{@{}c@{}} 
 $ \omega=e^{12}-e^{34}-e^{56}$ \\ $\rho=e^{135}-\frac{\sqrt{5}}{2} e^{146}+\frac{\sqrt{5}}{2}e^{236}+e^{245}+e^{246}  $ \\
\end{tabular} & -- \\ \hline
  $\mathfrak{g}_{19}$	& \begin{tabular}{@{}c@{}} 
 $ \omega=-e^{12}+e^{34}-e^{56}$ \\ $\rho=e^{135}+e^{146}-e^{236}+e^{245} $ \\
\end{tabular} & -- \\ \hline  
 $\mathfrak{g}_{20}$  & \begin{tabular}{@{}c@{}} 
 $ \omega=-e^{12}-e^{34}+e^{56}$ \\ $\rho=-e^{135}-e^{146}+e^{235}-e^{236}+e^{245}+e^{246} $ \\
\end{tabular} & -- \\ \hline   
$\mathfrak{g}_{21}$  & \begin{tabular}{@{}c@{}} 
 $ \omega=-e^{12}-e^{34}+e^{56}$ \\ $\rho=-2e^{136}+e^{145}+\frac{1}{2}e^{235}+e^{246} $\\
\end{tabular} & -- \\ \hline
$\mathfrak{g}_{22}$ & \begin{tabular}{@{}c@{}} 
 $ \omega=e^{16}+e^{23}+e^{45}$ \\ $\rho=e^{124}-e^{135}-e^{256}-e^{346} $ \\
\end{tabular} & \cmark \\ \hline
$\mathfrak{g}_{23}$ & \begin{tabular}{@{}c@{}}
 $ \omega=e^{12}+e^{34}+e^{56}$ \\ $\rho=2e^{136}+\frac{1}{2}e^{145}+e^{235}-e^{246} $ \\
\end{tabular} & -- \\ \hline  
 $\mathfrak{g}_{24}$  & \begin{tabular}{@{}c@{}} 
 $ \omega=-e^{16}+e^{25}-e^{34}$ \\ $\rho=-e^{123}+e^{145}+e^{246}+e^{356} $ \\
\end{tabular} & \cmark\\ \hline   
$\mathfrak{g}_{25}$  &  \begin{tabular}{@{}c@{}} 
 $ \omega=-e^{13}+e^{45}+e^{26}$ \\ $\rho=e^{156}+e^{124}-e^{235}-e^{346} $ \\
\end{tabular} & \cmark \\ \hline  
 $\mathfrak{g}_{26}$ & \begin{tabular}{@{}c@{}} 
 $ \omega=e^{16}+e^{23}-e^{36}+e^{45}$ \\ $\rho=-2e^{124}+e^{135}+e^{146}-e^{234}+e^{256} $ \\
\end{tabular} & -- \\ \hline    
 $\mathfrak{g}_{27}$  &  \begin{tabular}{@{}c@{}} 
 $ \omega=-\frac{\sqrt{3}}{2}e^{12}-e^{45}+e^{36}$ \\ $\rho=e^{135}+e^{146}+e^{234}+e^{235}-e^{256} $ \\
\end{tabular} & -- \\ \hline  
 $\mathfrak{g}_{28}$ & \begin{tabular}{@{}c@{}} 
 $ \omega=-e^{12}-e^{34}+e^{56}$ \\ $\rho=-e^{136}+e^{145}+e^{235}+e^{246} $ \\
\end{tabular} & \cmark \\ \hline     
$\mathfrak{g}_{29}$ & \begin{tabular}{@{}c@{}} 
 $ \omega=e^{13}+e^{24}-e^{56}$ \\ $\rho=e^{126}-e^{145}+e^{235}-e^{346} $ \\
\end{tabular} & \cmark\\ \hline     
 $\mathfrak{g}_{30}$ & \begin{tabular}{@{}c@{}} 
 $ \omega=e^{13}-e^{24}+e^{56}$ \\ $\rho=e^{125}-e^{126}+e^{145}+e^{146}+e^{236}+e^{345} $ \\
\end{tabular} & \cmark\\ \hline
$\mathfrak{g}_{31}$ 
& \begin{tabular}{@{}c@{}} 
 $ \omega=-e^{14}-e^{35}+e^{26}$ \\ $\rho=-e^{123}+e^{156}-e^{245}-e^{346} $ \\
\end{tabular}& \cmark \\ \hline             
 $\mathfrak{g}_{32}$ & \begin{tabular}{@{}c@{}} 
 $ \omega=-\sqrt{2} e^{13}-e^{24}-e^{56}$ \\ $\rho=-e^{125}+e^{146}-e^{236}+2e^{345}$ \\
\end{tabular}& \cmark \\ \hline        
$\mathfrak{g}_{33}$  & \begin{tabular}{@{}c@{}} 
 $ \omega=-e^{13}-e^{24}-e^{56}$ \\ $\rho=-e^{125}+e^{146}-e^{236}+e^{345} $ \\
\end{tabular} & \cmark\\ \hline        
\end{tabular} 
}
\end{table}
\renewcommand\arraystretch{1}

\begin{table}[H]
\begin{center}
\addtolength{\leftskip} {-3cm}
\addtolength{\rightskip}{-3cm}
\caption{$6$-dimensional unimodular  symplectic non-nilpotent solvable Lie algebras}
\label{table3}
\smallskip
\scalebox{0.75}{
\renewcommand\arraystretch{1.25}
\begin{tabular}{|c|c|c|}
\hline
$\mathfrak{g}$ & {\normalfont Structure constants} & Tamed closed $\text{SL}(3,\mathbb{C})$-structure \\ \hline
$\mathfrak{g}_{6,3}^{0,-1}$ &$(e^{26},e^{36},0,e^{46},-e^{56},0)$ & -- \\ \hline 
$\mathfrak{g}_{6,10}^{0,0}$ &$(e^{26},e^{36},0,e^{56},-e^{46},0)$& -- \\ \hline
$\mathfrak{g}_{6,13}^{-1,\frac{1}{2},0}$ &$(-\frac{1}{2}e^{16}+e^{23},-e^{26},\frac{1}{2}e^{36},e^{46},0,0)$& -- \\ \hline
$\mathfrak{g}_{6,13}^{\frac{1}{2},-1,0}$ &$(-\frac{1}{2}e^{16}+e^{23},\frac{1}{2}e^{26},-e^{36},e^{46},0,0)$& -- \\ \hline
$\mathfrak{g}_{6,15}^{-1}$ &$(e^{23},e^{26},-e^{36},e^{26}+e^{46},e^{36}-e^{56},0)$& -- \\ \hline
$\mathfrak{g}_{6,18}^{-1,-1}$  &$(e^{23},-e^{26},e^{36},e^{36}+e^{46},-e^{56},0)$& -- \\ \hline
$\mathfrak{g}_{6,21}^0$ & $(e^{23},0,e^{26},e^{46},-e^{56},0)$ & -- \\ \hline
$\mathfrak{g}_{6,36}^{0,0}$  &$(e^{23},0,e^{26},-e^{56},e^{46},0)$& -- \\ \hline
$\mathfrak{g}_{6,38}^0$  &$(e^{23},-e^{36},e^{26},e^{26}-e^{56},e^{36}+e^{46},0)$&  \begin{tabular}{@{}c@{}} 
 $\rho=-e^{124}-e^{135}+e^{236}-e^{456} $  \\ $ \Omega=-2e^{16}+e^{23}-e^{25}+e^{34}$ \\
\end{tabular}   \\ \hline
$\mathfrak{g}_{6,54}^{0,-1}$  &$(e^{16}+e^{35},-e^{26}+e^{45},e^{36},-e^{46},0,0)$& \begin{tabular}{@{}c@{}} 
 $\rho=e^{125}-e^{136}+e^{246}+e^{345} $  \\ $ \Omega=e^{14}+e^{23}+e^{34}+\frac{4}{3}e^{56}$ \\
\end{tabular}  \\ \hline
$\mathfrak{g}_{6,70}^{0,0}$  &$(-e^{26}+e^{35},e^{16}+e^{45},-e^{46},e^{36},0,0)$& -- \\ \hline
$\mathfrak{g}_{6,78}$  &$(-e^{16}+e^{25},e^{45},e^{24}+e^{36}+e^{46},e^{46},-e^{56},0)$& -- \\ \hline
$\mathfrak{g}_{6,118}^{0,-1,-1}$  &$(-e^{16}+e^{25},-e^{15}-e^{26},e^{36}-e^{45},e^{35}+e^{46},0,0)$& \begin{tabular}{@{}c@{}} 
 $\rho=e^{126}+e^{135}+e^{145}-e^{245}+e^{346} $  \\ $ \Omega=e^{14}+e^{23}+e^{56}$ \\
\end{tabular}  \\ \hline
$\mathfrak{n}_{6,84}^{\pm 1}$  &$(-e^{45},-e^{15}-e^{36},-e^{14}+e^{26} \mp e^{56},e^{56},-e^{46},0)$& -- \\ \hline
$\mathfrak{e}(2) \oplus \mathfrak{e}(2)$ &$(0,-e^{13},e^{12},0,-e^{46},e^{45})$ & -- \\ \hline 
$\mathfrak{e}(1,1) \oplus \mathfrak{e}(1,1)$ &$(0,-e^{13},-e^{12},0,-e^{46},-e^{45})$ & \begin{tabular}{@{}c@{}} 
 $\rho=-e^{125}-e^{126}+e^{135}-e^{145}-e^{246}+e^{345}+e^{346} $  \\ $ \Omega=-e^{14}+e^{23}-2e^{56}$ \\
\end{tabular} \\ \hline 
$\mathfrak{e}(2) \oplus \mathbb{R}^3$ &$(0,-e^{13},e^{12},0,0,0)$ & -- \\ \hline 
$\mathfrak{e}(1,1) \oplus \mathbb{R}^3$ &$(0,-e^{13},-e^{12},0,0,0)$ & -- \\ \hline 
$\mathfrak{e}(2) \oplus \mathfrak{e}(1,1)$ &$(0,-e^{13},e^{12},0,-e^{46},-e^{45})$ & -- \\ \hline 
$\mathfrak{e}(2) \oplus \mathfrak{h}$ &$(0,-e^{13},e^{12},0,0,e^{45})$ & -- \\ \hline 
$\mathfrak{e}(1,1) \oplus \mathfrak{h}$ &$(0,-e^{13},-e^{12},0,0,e^{45})$ & -- \\ \hline 
$A_{5,7}^{-1,\beta,-\beta}\oplus \mathbb{R}$ & $(e^{15},-e^{25},\beta e^{35},-\beta e^{45},0,0), \quad -1\leq \beta<0$ & \begin{tabular}{@{}c@{}} 
 $\rho=-e^{126}-e^{145}-e^{235}-e^{346} $  \\ $ \Omega=-e^{13}+e^{15}+e^{24}+e^{56}$ \\ 
 ($\beta=-1$) \\
\end{tabular}\\ \hline
$A_{5,8}^{-1} \oplus \mathbb{R}$  & $(e^{25},0,e^{35},-e^{45},0,0)$ & -- \\ \hline
$A_{5,13}^{-1,0,\gamma}\oplus \mathbb{R}$ & $(e^{15},-e^{25},\gamma e^{45},-\gamma e^{35},0,0), \quad \gamma>0$& -- \\ \hline
$A_{5,14}^0 \oplus \mathbb{R}$ & $(e^{25},0,e^{45},-e^{35},0,0)$& -- \\ \hline
$A_{5,15}^{-1} \oplus \mathbb{R}$  &$(e^{15}+e^{25},e^{25},-e^{35}+e^{45},-e^{45},0,0)$& -- \\ \hline
$A_{5,17}^{0,0,\gamma} \oplus \mathbb{R}$  &$(e^{25},-e^{15},\gamma e^{45},-\gamma e^{35},0,0), \quad -1< \gamma<0$& -- \\ \hline
$A_{5,17}^{0,0,-1} \oplus \mathbb{R}$  &$(e^{25},-e^{15},-e^{45}, e^{35},0,0)$&   \begin{tabular}{@{}c@{}} 
 $\rho=e^{135}-e^{146}+e^{236}+e^{245}+e^{346}-e^{356} $  \\ $ \Omega=e^{12}-e^{14}+e^{23}-e^{56}$ \\ 
\end{tabular} \\ \hline
$A_{5,17}^{\alpha,-\alpha,1} \oplus \mathbb{R}$  & $(\alpha e^{15} + e^{25},-e^{15} + \alpha e^{25},-\alpha e^{35} + e^{45},- e^{35}-\alpha e^{45},0,0), \quad \alpha>0$&  \begin{tabular}{@{}c@{}} 
 $\rho=e^{125}+e^{136}+e^{145}+e^{246}-e^{345} $  \\ $ \Omega=-e^{14}+e^{23}-e^{56}$ \\ 
\end{tabular} \\ \hline
$A_{5,18}^{0} \oplus \mathbb{R}$  &$(e^{25}+e^{35},-e^{15}+e^{45},e^{45},-e^{35},0,0)$& -- \\ \hline
$A_{5,19}^{-1,2} \oplus \mathbb{R}$  &$(-e^{15}+e^{23},e^{25},-2e^{35},2e^{45},0,0)$& -- \\ \hline
\end{tabular} 
}
\end{center}
\end{table}
\renewcommand\arraystretch{1}

\end{document}